\documentclass [12pt,a4paper,reqno]{amsart}
\usepackage{array}
\usepackage{tabularx}
\usepackage{enumerate}
\let\savedalpha\alpha
\usepackage{amsmath}
\usepackage{amsfonts}
\usepackage{amssymb}
\usepackage{latexsym}
\usepackage{bbm}
\usepackage[colorlinks=true,linkcolor=blue,citecolor=blue]{hyperref}
\usepackage{epsfig}
\let\alpha\savedalpha
\usepackage[bbgreekl]{mathbbol}
\numberwithin{equation}{section}
        
\newtheorem{definition}{Definition}[section]

\newtheorem{lemma}[definition]{Lemma}
\newtheorem{theorem}[definition]{Theorem}
\newtheorem{proposition}[definition]{Proposition}
\newtheorem{corollary}[definition]{Corollary}

\newtheorem{remark}[definition]{Remark}

\begin{document}

%%%%%%%%%%%%%%%%%%%%%%%%%%%%%%%%%%%%%%%%%%%%%%%%%%%%%%%%%%%%%%%%%%%%%%%%%%%%%%%
%%% quelques definitions utiles 
%\newcommand{\proof}{\paragraph{Proof}}
\newcommand{\fin}{$\Box$\\}
\newcommand{\C}{\mathbf{C}}
\newcommand{\R}{\mathbf{R}}
\newcommand{\Q}{\mathbf{Q}}
\newcommand{\N}{\mathbf{N}}
\renewcommand{\SS}{\mathbf{S}}
\newcommand{\PP}{\mathbf{P}}
\newcommand{\EE}{\mathbf{E}}
\newcommand{\B}{\hbox{$\cal B$}}
\setlength{\parindent}{0pt}
\newcommand{\ds}{\displaystyle}
\newcommand{\saut}[1]{\hfill\\[#1]}
\newcommand{\difrac}{\displaystyle \frac}
\newcommand{\dist}{\textrm{dist}}
\newcommand{\diam}{\mathrm{diam}}
\newcommand{\dimH}{\dim_{\mathcal{H}} }
\newcommand{\recov}{\textrm{Recov}}
\newcommand{\var}{\mathrm{Var}}
\newcommand{\gr}{\mathrm{Gr}}
\newcommand{\rg}{\mathrm{Rg}}
\newcommand{\mbf}{\textbf}
\newcommand{\levy}{\mathscr{B}}
\newcommand{\sheet}{\mathbbm{B}}
\newcommand{\sifbm}{\mathbf{B}}
\newcommand{\alphar}{\texttt{\large $\boldsymbol{\alpha}$}}
\newcommand{\alphalocbb}{\widetilde{\mathbb{\alpha}}} % Pour ArXiv
\newcommand{\alphainfbb}{\underline{\mathbb{\alpha}}} % Pour ArXiv
%%%%%%%%%%%%%%%%%%%%%%%%%%%%%%%%%%%%%%%%%%%%%%%%%%%%%%%%%%%%%%%%%%%%%%%%%%%%%%%

\title[Local regularity and Hausdorff dimension of Gaussian fields]{From almost sure local regularity to almost sure Hausdorff dimension for Gaussian fields 
}

\author{Erick Herbin}
\address{Ecole Centrale Paris, Grande Voie des Vignes, 92295 Chatenay-Malabry, France} \email{erick.herbin@ecp.fr}

\author{Benjamin Arras}
\email{benjamin.arras@student.ecp.fr}
\author{Geoffroy Barruel}
\email{geoffroy.barruel@student.ecp.fr}

\subjclass[2000]{60\,G\,15, 60\,G\,17, 60\,G\,10, 60\,G\,22, 60\,G\,60}

\begin{abstract}
Fine regularity of stochastic processes is usually measured in a local way by local H\"older exponents and in a global way by fractal dimensions. 
Following a previous work of Adler, we connect these two concepts for multiparameter Gaussian random fields. More precisely, we prove that almost surely the Hausdorff dimensions of the range and the graph in any ball $B(t_0,\rho)$ are bounded from above using the local H\"older exponent at $t_0$.
We define the deterministic local sub-exponent of Gaussian processes, which allows to obtain an almost sure lower bound for these dimensions.
Moreover, the Hausdorff dimensions of the sample path on an open interval are controlled almost surely by the minimum of the local exponents.

Then, we apply these generic results to the cases of the multiparameter fractional Brownian motion, the multifractional Brownian motion whose regularity function $H$ is irregular and the generalized Weierstrass function, whose Hausdorff dimensions were unknown so far.
\end{abstract}

\maketitle

\section{Introduction}

Since the 70's, the regularity of stochastic processes used to be considered in different ways. 
On one hand, the local regularity of sample paths is usually measured by local moduli of continuity and H\"older exponents (e.g. \cite{dudley, 2ml, Orey.Pruitt(1973), Yadrenko}). And on the other hand, the global regularity can be quantified by the global H\"older exponent (e.g. \cite{Xiao2009, Xiao2010}) or by fractal dimensions (Hausdorff dimension, box-counting dimension, packing dimension, \dots) and respective measures of the graph of the processes (e.g. \cite{Berman(1972), Pruitt1969, Strassen(1964)}).

As an example, if $B^H=\{B^H_t;\;t\in\R_+\}$ is a real-valued fractional Brownian motion (fBm) with self-similarity index $H\in (0,1)$, the pointwise H\"older exponent at any point $t\in\R_+$ satisfy $\alphar_{B^H}(t) = H$ almost surely.
Besides, the Hausdorff dimension of the graph of $B^H$ is given by $\dimH(\gr_{B^H}) = 2-H$ almost surely.
In this specific case, we observe a connection between the global and local points of view of regularity for fBm.
Is it possible to obtain some general result, for some larger class of processes?

\medskip

In \cite{Adler77}, Adler showed that the Hausdorff dimension of the graph of a $\R^d$-valued Gaussian field $X=\{X^{(i)}_t;\;1\leq i\leq p,\; t\in\R^N_+\}$, made of i.i.d. Gaussian coordinate processes $X^{(i)}$ with stationary increments, can be deduced from the local behavior of its incremental variance. More precisely, when the quantities $\sigma^2(t)=\EE[|X^{(i)}_{t+t_0}-X^{(i)}_{t_0}|^2]$ independent of $1\leq i\leq p$ and $t_0\in\R^N_+$ satisfy 
\begin{equation}\label{ineqAdler}
\forall \epsilon>0,\quad
|t|^{\alpha+\epsilon} \leq \sigma(t) \leq |t|^{\alpha-\epsilon}
\quad\textrm{as } t\rightarrow 0,
\end{equation}
the Hausdorff dimension of the graph $\gr_X=\{(t,X_t):t\in\R^N_+\}$ of $X$ is proved to be
\begin{align*}
%\dimH(\rg_X) &=  \\
\dimH(\gr_X) &= \min\left\{ \frac{N}{\alpha}, N+d (1-\alpha) \right\}.
\end{align*}
This result followed Yoder's previous works in \cite{Yoder} where the Hausdorff dimensions of the graph and also the range $\rg_X=\{ X_t: t\in\R^N_+ \}$ were obtained for a multiparameter Brownian motion in $\R^d$. 
As an application to Adler's result, the Hausdorff dimension of the graph of fractional Brownian motion can be deduced from the local H\"older exponents of its sample paths.
As an extension of this result, Xiao has completely determined in \cite{Xiao95} the Hausdorff dimensions of the image $X(K)$ and the graph $\gr_X(K)$ of a Gaussian field $X$ as previously, for a compact set $K\subset\R^N_+$, in function of $\dimH K$.

\medskip

In this paper, we aim at extending Adler's result to Gaussian random fields with non-stationary increments. We will see that this goal requires a localization of Adler's index $\alpha$ along the sample paths.
There is a large litterature about local regularity of Gaussian processes. We refer to \cite{AT, davar, ledouxtalagrand, marcusrosen} for a contemporary and detailled review of it.
This field of research is still very active, especially in the multiparameter context, and a non-exhaustive list of authors and recent works in this area includes Ayache \cite{AyLeVe, Ayache.Shieh.ea(2011)}, Mountford \cite{Baraka.Mountford.ea(2009)}, Dozzi \cite{dozzi07}, Khoshnevisan \cite{KX05}, Lawler \cite{lawler2011}, L\'evy V\'ehel \cite{2ml}, Lind \cite{lind08} and Xiao \cite{mwx, tudorxiao07, Xiao95, Xiao2009, Xiao2010}.

Usually the local regularity of an $\R^d$-valued stochastic process $X$ at $t_0\in\R^N_+$ is measured by the pointwise and local H\"older exponents $\alphar_X(t_0)$ and $\widetilde{\alphar}_X(t_0)$ defined by
\begin{align}
\alphar_X(t_0) &= \sup\left\{ \alpha>0: \limsup_{\rho\rightarrow 0}
\sup_{s,t\in B(t_0,\rho)} \frac{\|X_t-X_s\|}{\rho^{\alpha}} < +\infty \right\},\nonumber\\
\widetilde{\alphar}_X(t_0) &= \sup\left\{ \alpha>0: \lim_{\rho\rightarrow 0}
\sup_{s,t\in B(t_0,\rho)} \frac{\|X_t-X_s\|}{\| t-s \|^{\alpha}} < +\infty \right\}.\label{eq:localHolder-exp}
\end{align}

A general connection between the local structure of a stochastic process and the Hausdorff dimension of its graph has already been studied. In \cite{BCI03}, the specific case of local self-similarity property has been considered.
Here, we show how the local H\"older regularity of a Gaussian random field allows to estimate the Hausdorff dimensions of its range $\rg_X$ and its graph $\gr_X$.

Recently in \cite{2ml}, the quantities $\EE[|X_t - X_s|^2]$ when $s,t$ are close to $t_0\in\R^N_+$ are proved to capture a lot of informations about the almost sure local regularity. More precisely, the almost sure $2$-microlocal frontier of $X$ at $t_0$ allows to predict the evolution of the local regularity at $t_0$ under fractional integrations or derivations. Particularly, as special points of the $2$-microlocal frontier, both pointwise and local H\"older exponents can be derived from the study of $\EE[|X_t - X_s|^2]$.
For all $t_0\in\R_+^N$, we define in Section~\ref{sec:subexp} the exponents $\alphainfbb_X(t_0)$ and $\alphalocbb_X(t_0)$ of a real-valued Gaussian process $X$ as the minimum of $\alphainfbb>0$ and maximum of $\alphalocbb>0$ such that 
\begin{equation*}
\forall s,t\in B(t_0,\rho_0),\quad
\| t-s \|^{2\, \alphainfbb} \leq \EE[|X_t-X_s|^2]
\leq \| t-s \|^{2\, \alphalocbb},
\end{equation*}
for some $\rho_0>0$.
The exponents of the components $X^{(i)}$ of a Gaussian random field $X=(X^{(1)},\dots,X^{(d)})$ allow to get almost sure lower and upper bounds for quantities, 
$$\lim_{\rho\rightarrow 0}\dimH(\gr_X(B(t_0,\rho)))\quad\textrm{and}\quad\lim_{\rho\rightarrow 0}\dimH(\rg_X(B(t_0,\rho))).$$
After the statement of the main result in Section~\ref{sec:main}, the almost sure local Hausdorff dimensions are given uniformly in $t_0\in\R_+^N$ and the global dimensions $\dimH(\gr_X(I))$ and $\dimH(\rg_X(I))$ are almost surely bounded for any open interval $I\subset\R_+^N$, in function of $\inf_{t\in I}\alphainfbb_{X^{(i)}}(t)$ and $\inf_{t\in I}\alphalocbb_{X^{(i)}}(t)$.
Sections~\ref{sec:up} and \ref{sec:low} are devoted to the proofs of the upper bound and lower bound of the Hausdorff dimensions respectively.

\medskip

In Section~\ref{sec:app}, the main result is applied to some stochastic processes whose increments are not stationary and whose Hausdorff dimension is still unknown.

The first one is the multiparameter fractional Brownian motion (MpfBm), derived from the set-indexed fractional Brownian motion introduced in \cite{sifBm, MpfBm}. On the contrary to fractional Brownian sheet studied in \cite{AyXiao, WuXiao}, the MpfBm does not satisfy the increment stationarity property. Then the study of the local regularity of its sample path allows to determine the Hausdorff dimension of its graph in Section~\ref{sec:mpfbm}.

The second application is the multifractional Brownian motion (mBm), introduced in \cite{RPJLV,BJR} as an extension of the classical fractional Brownian motion where the self-similarity index $H\in (0,1)$ is substituted with a function $H:\R_+\rightarrow (0,1)$ in order to allow the local regularity to vary along the sample path.
The immediate consequence is the loss of the increment stationarity property. Then, the knowledge of local H\"older regularity implies the Hausdorff dimensions of the graph and the range of the mBm.
In the case of a regular function $H$, the almost sure value of $\lim_{\rho\rightarrow 0}\dimH(\gr_X(B(t_0,\rho)))$ was already known to be $2-H(t_0)$ for any fixed $t_0\in\R_+$. In Section~\ref{sec:mbm}, this almost sure result is proved uniformly in $t_0$. The new case of an irregular function $H$ is also considered.

The last application of this article concerns the generalized Weierstrass function, defined as a stochastic Gaussian version of the well-known Weierstrass function, where the index varies along the trajectory. The local H\"older regularity is determined in Section~\ref{sec:GW} and consequentely, the Hausdorff dimension of its sample path.

\section{Hausdorff dimension of the sample paths of Gaussian random fields}

In this paper, we denote by {\em multiparameter Gaussian random field} in $\R^d$, a stochastic process $X=\{ X_t;\;t\in\R_+^N \}$, where $X_t = (X^{(1)}_t,\dots,X^{(d)}_t)\in\R^d$ for all $t\in\R^N_+$ and the coordinate processes $X^{(i)}=\{ X^{(i)}_t;\;t\in\R^N_+\}$ are independent real-valued Gaussian processes with the same law.

\subsection{A new local exponent}\label{sec:subexp}

According to \cite{2ml}, the local regularity of a Gaussian process $X=\{X_t;\;t\in\R^N_+\}$ can be obtained by the {\em deterministic local H\"older exponent}
\begin{equation}\label{DetLocalHolder}
\alphalocbb_X(t_0) = \sup\left\{ \alpha>0 :
\lim_{\rho\rightarrow 0} \sup_{s,t\in B(t_0,\rho)} \frac{\EE[|X_t-X_s|^2]}{\|t-s\|^{2\alpha}}
<+\infty \right\}.
\end{equation}
More precisely, the local H\"older exponent of $X$ at any $t_0\in\R_+^N$ is proved to satisfy $\widetilde{\alphar}_X(t_0)=\alphalocbb_X(t_0)$ a.s.

In order to get a localized version of (\ref{ineqAdler}), we need to introduce a new exponent $\alphainfbb_X(t_0)$,  the {\em deterministic local sub-exponent} at any $t_0\in\R_+^N$,
\begin{align}\label{DetUpLocalHolder}
\alphainfbb_X(t_0) &= \inf\left\{ \alpha>0 :
\lim_{\rho\rightarrow 0} \inf_{s,t\in B(t_0,\rho)} \frac{\EE[|X_t-X_s|^2]}{\|t-s\|^{2\alpha}}
=+\infty \right\} \\
&= \sup\left\{ \alpha>0 :
\lim_{\rho\rightarrow 0} \inf_{s,t\in B(t_0,\rho)} \frac{\EE[|X_t-X_s|^2]}{\|t-s\|^{2\alpha}}
=0 \right\}. \nonumber
\end{align}

As usually, this double definition relies on the equality
\begin{align*}
\frac{\EE[|X_t-X_s|^2]}{\|t-s\|^{2 \alpha'}}
= \frac{\EE[|X_t-X_s|^2]}{\|t-s\|^{2 \alpha}}
\times \| t-s \|^{2(\alpha-\alpha')}.
\end{align*}

\

\begin{lemma}\label{lemcovinc}
Let $X=\{X_t;\;t\in\R^N_+\}$ be a multiparameter Gaussian process.\\
Consider $\alphalocbb_X(t_0)$ and $\alphainfbb_X(t_0)$ the deterministic local H\"older exponent and local sub-exponent of $X$ at $t_0\in\R_+^N$ (as defined in (\ref{DetLocalHolder}) and (\ref{DetUpLocalHolder})).

For any $\epsilon>0$, there exists $\rho_0>0$ such that
\begin{equation*}
\forall s,t\in B(t_0,\rho_0),\quad
\| t-s \|^{2\, \alphainfbb_X(t_0) +\epsilon} \leq \EE[|X_t-X_s|^2]
\leq \| t-s \|^{2\, \alphalocbb_X(t_0) -\epsilon}.
\end{equation*}
\end{lemma}

\ 

\begin{proof}
For any $\epsilon >0$, the definition of $\alphalocbb_X(t_0)$ leads to
\begin{equation*}
\lim_{\rho\rightarrow 0} \sup_{s,t\in B(t_0,\rho)} \frac{\EE[|X_t-X_s|^2]}{\|t-s\|^{2\,\alphalocbb_X(t_0) - \epsilon}} =0.
\end{equation*}
Then there exits $\rho_1>0$ such that
\begin{equation*}
0<\rho\leq\rho_1 \Rightarrow \forall s,t\in B(t_0,\rho),\ \EE[|X_t-X_s|^2] \leq \|t-s\|^{2\,\alphalocbb_X(t_0) - \epsilon}
\end{equation*}
and then
\begin{equation*}
\forall s,t\in B(t_0,\rho_1),\quad \EE[|X_t-X_s|^2] \leq \|t-s\|^{2\,\alphalocbb_X(t_0) - \epsilon}.
\end{equation*}

\vspace{10pt}
For the lower bound, we use the definition of the new exponent $\alphainfbb_X(t_0)$
\begin{align*}
\lim_{\rho\rightarrow 0} \inf_{s,t\in B(t_0,\rho)} \frac{\EE[|X_t-X_s|^2]}{\|t-s\|^{2\,\alphainfbb_X(t_0) + \epsilon}} =+\infty.
\end{align*}
Then, there exists $\rho_2>0$ such that
\begin{equation*}
0<\rho\leq\rho_2 \Rightarrow \forall s,t\in B(t_0,\rho),\ \EE[|X_t-X_s|^2] \geq \|t-s\|^{2\,\alphainfbb_X(t_0) + \epsilon}
\end{equation*}
and then
\begin{equation*}
\forall s,t\in B(t_0,\rho_2),\quad 
\EE[|X_t-X_s|^2] \geq \|t-s\|^{2\,\alphainfbb_X(t_0) + \epsilon}.
\end{equation*}
The result follows setting $\rho_0=\rho_1\wedge\rho_2$.
\end{proof}

\

From the previous result, we can derive an ordering relation between the deterministic local sub-exponent and the deterministic local H\"older exponent.
We have
\begin{equation}\label{ineqexp}
\forall t_0\in\R^N_+,\quad
\alphalocbb_X(t_0) \leq \alphainfbb_X(t_0).
\end{equation}

\subsection{Main result: The Hausdorff dimension of Gaussian random fields}\label{sec:main}

For sake of self-containess of the paper, we recall the basic frame of the Hausdorff dimension definition.

For all $\delta>0$, we denote by $\delta$-covering of a non-empty subset $E$ of $\R^d$. all collection $A = (A_i)_{i\in\N}$ such that
\begin{itemize}
%\item $\forall i\in\N, A_i\subset E$;
\item $\forall i \in \N, \diam(A_i) < \delta$, where $\diam(A_i)$ denotes $\sup(\|x-y\|;\; x,y\in A_i)$ ; and
\item $E \subseteq \bigcup_{i \in \N}A_i$. 
\end{itemize}
We denote by $\Sigma_\delta(E)$ the set of $\delta$-covering de $E$ and by $\Sigma(E)$ the set of the covering of $E$. We define
$$\mathcal{H}^s_{\delta}(E)=\inf_{A\in{\Sigma_\delta(E)}}\left\{\sum_{i=1}^{\infty}\diam(A_i)^s\right\},$$
and the Hausdorff measure of $E$ by
\begin{align*}
\mathcal{H}^s(E)=\lim_{\delta\rightarrow 0}\mathcal{H}^s_{\delta}(E)
= \begin{cases}
+\infty & \text{si } 0 \leq s < \dimH(E), \\
0 & \text{si } s > \dimH(E).
\end{cases}
\end{align*}
The quantity $\dimH(E)$ is the Hausdorff dimension of $E$. It is defined by 
$$\dimH(E)=\inf \left\{s \in \R_+: \mathcal{H}^s(E)=0\right\}=\sup\left\{s \in \R_+: \mathcal{H}^s(E)=+\infty\right\}.$$

\

For any random field $X=\{X^{(i)}_t;\;1\leq i\leq p,\;t\in\R^N_+\}$ made of i.i.d. Gaussian coordinate processes with possibly non-stationary increments, the Hausdorff dimensions of the range $\rg_X(B(t_0,\rho)) = \{ X_t;\; t\in B(t_0,\rho)\}$ and the graph $\gr_X(B(t_0,\rho)) = \{ (t,X_t);\; t\in B(t_0,\rho)\}$ of $X$ in the ball $B(t_0,\rho)$ of center $t_0$ and radius $\rho>0$ can be estimated when $\rho$ goes to $0$, using the deterministic local H\"older exponent and the deterministic local sub-exponent of $X^{(i)}$ at $t_0$.

In the following statements and in the sequel of the paper, the deterministic local H\"older exponent $\alphalocbb_{X^{(i)}}(t_0)$ and the deterministic local sub-exponent $\alphainfbb_{X^{(i)}}(t_0)$ of $X^{(i)}$ at any $t_0\in\R^N_+$ are independent of $1\leq i\leq d$, since the component $X^{(i)}$ are assumed to be i.i.d.

\medskip

\begin{theorem}[Pointwise almost sure result]\label{thmain}
Let $X=\{X_t;\;t\in\R^N_+\}$ be a multi-parameter Gaussian random field in $\R^d$. Let $\alphalocbb_{X^{(i)}}(t_0)$ be the deterministic local H\"older exponent and $\alphainfbb_{X^{(i)}}(t_0)$ the deterministic local sub-exponent of $X^{(i)}$ at $t_0\in\R_+^N$ as defined in (\ref{DetLocalHolder}) and (\ref{DetUpLocalHolder}), independent of $1\leq i\leq d$.
Assume that $\alphalocbb_{X^{(i)}}(t_0)>0$.

%% Pointwise almost sure result
Then, the Hausdorff dimensions of the graph and the range of $X$ satisfy almost surely,
\begin{align*}
\left. \begin{array}{l r}
\textrm{if } N\leq d\ \alphainfbb_{X^{(i)}}(t_0), 
& N/\alphainfbb_{X^{(i)}}(t_0) \\
\textrm{if } N> d\ \alphainfbb_{X^{(i)}}(t_0), 
& N + d(1-\alphainfbb_{X^{(i)}}(t_0))
\end{array} \right\} 
\leq &\lim_{\rho\rightarrow 0}\dimH(\gr_X(B(t_0,\rho))) \\
&\leq \min\left\{\frac{N}{\alphalocbb_{X^{(i)}}(t_0)} ; N + d(1-\alphalocbb_{X^{(i)}}(t_0))\right\}
\end{align*}
and
\begin{align*}
\left. \begin{array}{l r}
\textrm{if } N\leq d\ \alphainfbb_{X^{(i)}}(t_0), 
& N/\alphainfbb_{X^{(i)}}(t_0) \\
\textrm{if } N> d\ \alphainfbb_{X^{(i)}}(t_0), 
& d
\end{array} \right\}
\leq \lim_{\rho\rightarrow 0}\dimH(\rg_X(B(t_0,\rho))) 
\leq \min\left\{\frac{N}{\alphalocbb_{X^{(i)}}(t_0)} ; d\right\}.
\end{align*}
\end{theorem}

The proof of Theorem \ref{thmain} relies on Propositions \ref{propmajdimH} and \ref{propmindimH}.

\

\begin{theorem}[Uniform almost sure result]\label{thmainunif}
Let $X=\{X_t;\;t\in\R^N_+\}$ be a multi-parameter Gaussian random field in $\R^d$. Let $\alphalocbb_{X^{(i)}}(t)$ be the deterministic local H\"older exponent and $\alphainfbb_{X^{(i)}}(t)$ the deterministic local sub-exponent of $X^{(i)}$ at any $t\in\R_+^N$.

%% Uniform almost sure result
Set $\mathcal{A} = \{ t\in\R_+^N: \liminf_{u\rightarrow t}\alphalocbb_{X^{(i)}}(u)>0\}$. 

Then, with probability one, for all $t_0\in\mathcal{A}$, 
\begin{itemize}
\item if $N\leq d\ \liminf_{u\rightarrow t_0}\alphainfbb_{X^{(i)}}(u)$ then
\begin{align*}
\frac{N}{\displaystyle\liminf_{u\rightarrow t_0}\alphainfbb_{X^{(i)}}(u)} 
\leq \lim_{\rho\rightarrow 0}&\dimH(\gr_X(B(t_0,\rho))) \\
&\leq \min\left\{\frac{N}{\displaystyle\liminf_{u\rightarrow t_0}\alphalocbb_{X^{(i)}}(u)} ; N + d(1-\liminf_{u\rightarrow t_0}\alphalocbb_{X^{(i)}}(u))\right\}
\end{align*}
and
\begin{align*}
\frac{N}{\displaystyle\liminf_{u\rightarrow t_0}\alphainfbb_{X^{(i)}}(u)}
\leq \lim_{\rho\rightarrow 0}&\dimH(\rg_X(B(t_0,\rho))) 
\leq \min\left\{\frac{N}{\displaystyle\liminf_{u\rightarrow t_0}\alphalocbb_{X^{(i)}}(u)} ; d\right\}.
\end{align*}

\item if $N> d\ \liminf_{u\rightarrow t_0}\alphainfbb_{X^{(i)}}(u)$ then
\begin{align*}
N + d(1-\liminf_{u\rightarrow t_0}\alphainfbb_{X^{(i)}}(u))
\leq \lim_{\rho\rightarrow 0}&\dimH(\gr_X(B(t_0,\rho))) \\
&\leq \min\left\{\frac{N}{\displaystyle\liminf_{u\rightarrow t_0}\alphalocbb_{X^{(i)}}(u)} ; N + d(1-\liminf_{u\rightarrow t_0}\alphalocbb_{X^{(i)}}(u))\right\}
\end{align*}
and
\begin{align*}
\lim_{\rho\rightarrow 0}\dimH(\rg_X(B(t_0,\rho))) = d.
\end{align*}

\end{itemize}
\end{theorem}

The proof of Theorem \ref{thmainunif} relies on Proposition \ref{propmajdimH} and Corollary \ref{cormindimHunif2}.

\begin{theorem}[Global almost sure result]\label{thmaincompact}
Let $X=\{X_t;\;t\in\R^N_+\}$ be a multiparameter Gaussian field in $\R^d$. Let $\alphalocbb_{X^{(i)}}(t)$ be the deterministic local H\"older exponent and $\alphainfbb_{X^{(i)}}(t)$ the deterministic local sub-exponent of $X^{(i)}$ at any $t\in\R_+^N$.

For any open interval $I \subset\R^N_+$, assume that the quantities $\alphainfbb = \inf_{t\in I}\alphainfbb_{X^{(i)}}(t)$ and $\alphalocbb = \inf_{t\in I}\alphalocbb_{X^{(i)}}(t)$ satisfy
$0 < \alphalocbb \leq \alphainfbb$. 
Then, with probability one,
\begin{align*}
\left. \begin{array}{l r}
\textrm{if } N\leq d\ \alphainfbb, 
& N/\alphainfbb \\
\textrm{if } N> d\ \alphainfbb, 
& N + d(1-\alphainfbb)
\end{array} \right\} 
\leq \dimH(\gr_X(I))
\leq \min\left\{N/\alphalocbb ; N + d(1-\alphalocbb) \right\}
\end{align*}
and
\begin{align*}
\left. \begin{array}{l r}
\textrm{if } N\leq d\ \alphainfbb, 
&N/\alphainfbb \\
\textrm{if } N> d\ \alphainfbb, 
& d
\end{array} \right\}
\leq \dimH(\rg_X(I)) 
\leq \min\left\{N/\alphalocbb ; d\right\}.
\end{align*}

\end{theorem}

The proof of Theorem \ref{thmaincompact} relies on Corollary \ref{cormajdimHunif1} and Corollary \ref{cormindimHunif1}.

\subsection{Upper bound for the Hausdorff dimension}\label{sec:up}

\begin{lemma}\label{lemdimHmaj}
Let $X=\{X_t;\;t\in\R^N_+\}$ be a multiparameter random process with values in $\R^d$. Let $\widetilde{\alphar}_X(t_0)$ be the local H\"older exponent of $X$ at $t_0\in\R_+^N$.

For any $\omega$ such that $\widetilde{\alphar}_X(t_0)>0$, 
\begin{align*}
\lim_{\rho\rightarrow 0}\dimH(\rg_X(B(t_0,\rho))) \leq \lim_{\rho\rightarrow 0}\dimH&(\gr_X(B(t_0,\rho))) \\
&\leq \min\left\{ \frac{N}{\widetilde{\alphar}_X(t_0)} ; N + d(1-\widetilde{\alphar}_X(t_0))\right\}.
\end{align*}
\end{lemma}

\begin{proof}
The first inequality follows the fact that the range $\rg_X(B(t_0,\rho))$ is a projection of the graph $\gr_X(B(t_0,\rho))$.
For the second inequality, we need to localize the argument of Yoder (\cite{Yoder}), who proved the upper bound for the Hausdorff dimensions of the range and the graph of a H\"olderian function from $\R^N$ (or $[0,1]^N$) to $\R^d$ (see also \cite{falconer}, Corollary~11.2 p. 161).

Assume that $\omega$ is fixed such that $\widetilde{\alphar}_X(t_0,\omega)>0$.
By definition of $\widetilde{\alphar}_X(t_0)$, for all $\epsilon>0$ there exists $\rho_0>0$ such that for all $\rho\in (0,\rho_0]$,
\begin{align*}
\forall s,t\in B(t_0,\rho),\quad
\|X_t(\omega)-X_s(\omega)\| \leq \|t-s\|^{\widetilde{\boldsymbol{\alpha}}_X(t_0,\omega)-\epsilon}.
\end{align*}
There exists a real $0<\delta_0<1$ such that for all $u\in [0,1]^N$, $t_0 + \delta_0.u\in B(t_0,\rho_0)$ and consequently,
\begin{align*}
\forall u,v\in [0,1]^N,\quad
\|X_{t_0+\delta_0.u}(\omega)-X_{t_0+\delta_0.v}(\omega)\| \leq (\delta_0\ \|u-v\|)^{\widetilde{\boldsymbol{\alpha}}_X(t_0,\omega)-\epsilon}.
\end{align*}

Then, the function $Y_{\bullet}(\omega):u\mapsto Y_u(\omega)=X_{t_0+\rho_0.u}(\omega)$ is H\"older-continuous of order $\widetilde{\alphar}_X(t_0,\omega)-\epsilon$ on $[0,1]^N$ and therefore, according to \cite{Yoder},
\begin{align*}
\dimH(\rg_{Y_{\bullet}(\omega)}([0,1]^N)) \leq \dimH&(\gr_{Y_{\bullet}(\omega)}([0,1]^N)) \\
&\leq \min\left\{ \frac{N}{\widetilde{\alphar}_X(t_0,\omega)-\epsilon} ; N + d(1-\widetilde{\alphar}_X(t_0,\omega)+\epsilon)\right\}.
\end{align*}
We can observe that the graph $\gr_{X_{\bullet}(\omega)}(t_0+\delta_0.[0,1]^N))$ is an affine transformation of the graph $\gr_{Y_{\bullet}(\omega)}([0,1]^N))$, therefore their Hausdorff dimensions are equal.
Moreover, there exists $\rho>0$ such that $B(t_0,\rho) \subset t_0+\delta_0.[0,1]^N$. By monotony of the function $\rho\mapsto\dimH(\gr_{X_{\bullet}(\omega)}(B(t_0,\rho)))$, we can write
\begin{align*}
\lim_{\rho\rightarrow 0}\dimH&(\gr_{X_{\bullet}(\omega)}(B(t_0,\rho))) 
\leq \min\left\{ \frac{N}{\widetilde{\alphar}_X(t_0,\omega)-\epsilon} ; N + d(1-\widetilde{\alphar}_X(t_0,\omega)+\epsilon)\right\}.
\end{align*}
Since this inequality stands for all $\epsilon>0$, we get
\begin{align*}
\lim_{\rho\rightarrow 0}\dimH&(\gr_{X_{\bullet}(\omega)}(B(t_0,\rho))) 
\leq \min\left\{ \frac{N}{\widetilde{\alphar}_X(t_0,\omega)} ; N + d(1-\widetilde{\alphar}_X(t_0,\omega))\right\}.
\end{align*}
\end{proof}

Lemma \ref{lemdimHmaj} gives a random upper bound for the Hausdorff dimensions of the (localized) range and graph of the sample path, in function of its local H\"older exponents. 
When $X$ is a multiparameter Gaussian field in $\R^d$, we prove that this upper bound can be expressed almost surely with the deterministic local H\"older exponent of the Gaussian component processes $X^{(i)}$.

\begin{proposition}\label{propmajdimH}
Let $X=\{X_t;\;t\in\R^N_+\}$ be a multiparameter Gaussian field in $\R^d$. Let $\alphalocbb_{X^{(i)}}(t_0)$ be the deterministic local H\"older exponent of $X^{(i)}$ at $t_0\in\R_+^N$ and assume that $\alphalocbb_{X^{(i)}}(t_0)>0$.

Then, almost surely
\begin{align*}
\lim_{\rho\rightarrow 0}\dimH(\rg_X(B(t_0,\rho))) \leq \lim_{\rho\rightarrow 0}\dimH&(\gr_X(B(t_0,\rho))) \\
&\leq \min\left\{N/\alphalocbb_{X^{(i)}}(t_0) ; N + d(1-\alphalocbb_{X^{(i)}}(t_0))\right\}.
\end{align*}

Moreover, an uniform result can be stated on the set 
$$\mathcal{A}=\{t_0\in\R^N_+: \liminf_{u\rightarrow t_0}\alphalocbb_{X^{(i)}}(u)>0\}.$$ With probability one, for all $t_0\in \mathcal{A}$, 
\begin{align*}
\lim_{\rho\rightarrow 0}\dimH(\rg_X(B(t_0,\rho))) \leq &\lim_{\rho\rightarrow 0}\dimH(\gr_X(B(t_0,\rho))) \\
&\quad\leq \min\left\{N/\liminf_{u\rightarrow t_0}\alphalocbb_{X^{(i)}}(u) ; N + d(1-\liminf_{u\rightarrow t_0}\alphalocbb_{X^{(i)}}(u))\right\}.
\end{align*}

\end{proposition}

\begin{proof}
In \cite{2ml}, the local H\"older exponent of any Gaussian process $Y$ at $t_0\in\R^N_+$ such that $\alphalocbb_Y(t_0)>0$ is proved to satisfy $\widetilde{\alphar}_Y(t_0) = \alphalocbb_Y(t_0)$ almost surely.
Therefore, by definition of $\widetilde{\alphar}_{X^{(i)}}(t_0)$, for all $\epsilon>0$ there exists $\rho_0>0$ such that for all $\rho\in (0,\rho_0]$, we have almost surely
\begin{align*}
\forall s,t\in B(t_0,\rho),\quad
|X^{(i)}_t-X^{(i)}_s| \leq \|t-s\|^{\alphalocbb_{X^{(i)}}(t_0)-\epsilon},
\end{align*}
and consequently, almost surely
\begin{align}\label{eqmajGaussField}
\forall s,t\in B(t_0,\rho),\quad
\|X_t-X_s\| \leq K\ \|t-s\|^{\alphalocbb_{X^{(i)}}(t_0)-\epsilon},
\end{align}
for some constant $K>0$.

From (\ref{eqmajGaussField}), we deduce that $\widetilde{\alphar}_X(t_0)\geq \alphalocbb_{X^{(i)}}(t_0)$ almost surely.
Then Lemma \ref{lemdimHmaj} implies almost surely
\begin{align*}
\lim_{\rho\rightarrow 0}\dimH(\rg_X(B(t_0,\rho))) \leq \lim_{\rho\rightarrow 0}\dimH&(\gr_X(B(t_0,\rho))) \\
&\leq \min\left\{N/\alphalocbb_{X^{(i)}}(t_0) ; N + d(1-\alphalocbb_{X^{(i)}}(t_0))\right\}.
\end{align*}

For the uniform result on $t_0\in\R^N_+$, we use the Theorem 3.14 of \cite{2ml} which states that if $Y$ is a Gaussian process such that the function $t_0\mapsto\liminf_{u\rightarrow t_0}\alphalocbb_Y(u)$ is positive, then with probability one,
\begin{align*}
\forall t_0\in\R^N_+,\quad
\liminf_{u\rightarrow t_0}\alphalocbb_Y(u) 
\leq \widetilde{\alphar}_Y(t_0)
\leq \limsup_{u\rightarrow t_0}\alphalocbb_Y(u).
\end{align*} 
This inequality yields to the existence of $\Omega_i\in\mathcal{F}$ for all $1\leq i\leq d$ with $\PP(\Omega_i)=1$ and: \\
For all $\omega\in\Omega_i$, all $t_0\in \mathcal{A}$ and all $\epsilon>0$, there exists $\rho_0>0$ such that for all $\rho\in (0,\rho_0]$,
\begin{align*}
\forall s,t\in B(t_0,\rho),\quad
|X^{(i)}_t(\omega)-X^{(i)}_s(\omega)| \leq \|t-s\|^{\liminf_{u\rightarrow t_0}\alphalocbb_{X^{(i)}}(u)-\epsilon}.
\end{align*}
This yields to: For all $\omega\in\bigcap_{1\leq i\leq d}\Omega_i$, all $t_0\in \mathcal{A}$ and all $\epsilon>0$, there exists $\rho_0>0$ such that for all $\rho\in (0,\rho_0]$,
\begin{align*}
\forall s,t\in B(t_0,\rho),\quad
\|X_t(\omega)-X_s(\omega)\| \leq K\ \|t-s\|^{\liminf_{u\rightarrow t_0}\alphalocbb_{X^{(i)}}(u)-\epsilon},
\end{align*}
for some constant $K>0$.

With the argument of Lemma \ref{lemdimHmaj}, we deduce 
\begin{align*}
\lim_{\rho\rightarrow 0}\dimH(\rg_X(B(t_0,\rho),\omega)) &\leq \lim_{\rho\rightarrow 0}\dimH(\gr_X(B(t_0,\rho),\omega)) \\
&\quad\leq \min\left\{N/\liminf_{u\rightarrow t_0}\alphalocbb_{X^{(i)}}(u) ; N + d(1-\liminf_{u\rightarrow t_0}\alphalocbb_{X^{(i)}}(u))\right\},
\end{align*}
which is the result stated.
\end{proof}

\medskip

\begin{corollary}\label{cormajdimHunif1}
Let $X=\{X_t;\;t\in\R^N_+\}$ be a multiparameter Gaussian field in $\R^d$ and $\alphalocbb_{X^{(i)}}(t_0)$ the deterministic local H\"older exponent of $X^{(i)}$ at $t_0\in\R_+^N$.\\
Assume that for some bounded interval $I\subset\R^N_+$, %on the form $I=[a,b]$ or $I=(a,b)$, 
we have
$\alpha = \inf_{t_0\in I}\alphalocbb_{X^{(i)}}(t_0) >0$. 
Then, with probability one,
$$\dimH(\rg_X(I)) \leq \dimH(\gr_X(I))\leq \min\left\{N/\alpha ; N + d(1-\alpha)\right\}.$$

\end{corollary}

\begin{proof}
With the same arguments as in the proof of Proposition \ref{propmajdimH}, we can claim that, with probability one,
$\forall t_0\in I,\ \alpha\leq\widetilde{\alphar}_X(t_0).$
Then, there exists $\Omega_0\in\mathcal{F}$ with $\PP(\Omega_0)=1$ and:
For all $\omega\in\Omega_0$, all $t_0\in I$ and all $\epsilon>0$, there exist $\rho_0>0$ and $K>0$ such that $\forall \rho\in (0,\rho_0]$,
\begin{align*}
\forall s,t\in B(t_0,\rho),\quad
\| X_t(\omega)-X_s(\omega) \| \leq K\ \| t-s \|^{\alpha-\epsilon}.
\end{align*}

Then the continuity of $t\mapsto X_t(\omega)$ on the bounded interval $I$ allows to deduce that, for all $\omega\in\Omega_0$ and all $\epsilon>0$, there exists a constant $K'>0$ such that
\begin{align}\label{eq:maj-holder-global}
\forall s,t\in I,\quad
\| X_t(\omega)-X_s(\omega) \| \leq K'\ \| t-s \|^{\alpha-\epsilon}.
\end{align}

If the interval $I$ is compact, we can exhibit an affine one-to-one mapping $I\rightarrow [0,1]^N$ and conclude with the arguments of Lemma \ref{lemdimHmaj} that \cite{Yoder} implies
\begin{align*}
\dimH(\rg_{X_{\bullet}(\omega)}(I)) \leq \dimH(\gr_{X_{\bullet}(\omega)}(I)) 
\leq \min\left\{ \frac{N}{\alpha-\epsilon} ; N + d(1-\alpha+\epsilon)\right\}
\qquad\textrm{a.s.}
\end{align*}
Since this inequality stands for any $\epsilon>0$, the result follows in that case.

If $I$ is not closed, we remark that 
$$\dimH(\rg_{X_{\bullet}(\omega)}(I)) \leq \dimH(\rg_{X_{\bullet}(\omega)}(\overline{I}))\quad\textrm{and}\quad \dimH(\gr_{X_{\bullet}(\omega)}(I)) \leq \dimH(\gr_{X_{\bullet}(\omega)}(\overline{I})).$$
Then, extending the inequality (\ref{eq:maj-holder-global}) to $\overline{I}$ by continuity, the result for the compact interval $\overline{I}$ is proved as previously.
\end{proof}

%\begin{corollary}\label{cormajdimHunif2}
%Let $X=\{X_t;\;t\in\R^N_+\}$ be a multiparameter Gaussian field in $\R^d$ and $\alphalocbb_{X^{(i)}}(t)$ the deterministic local H\"older exponent of $X^{(i)}$ at any $t\in\R_+^N$.\\
%Set $\mathcal{A} = \{ t\in\R_+^N: \liminf_{u\rightarrow t}\alphalocbb_{X^{(i)}}(u)>0\}$. 

%Then, with probability one, for all $t_0\in\mathcal{A}$,
%\begin{align*}
%\lim_{\rho\rightarrow 0}\dimH(\rg_X(B(t_0,\rho))) \leq \lim_{\rho\rightarrow 0}&\dimH(\gr_X(B(t_0,\rho))) \\
%&\leq \min\left\{N/\liminf_{u\rightarrow t}\alphalocbb_{X^{(i)}}(u) ; N + d(1-\liminf_{u\rightarrow t}\alphalocbb_{X^{(i)}}(u))\right\}.
%\end{align*}

%\end{corollary}

%{\bf Attention: Already proved in Proposition 2.6.}

%\begin{proof}
%Corollary \ref{cormajdimHunif1} implies the existence of $\Omega^*\in\mathcal{F}$ with $P(\Omega^*)=1$ such that: For all $\omega\in\Omega^*$ and all $a,b\in\Q_+^N$ with $a\prec b$, such that $\alpha=\inf_{t_0\in [a,b]}\alphalocbb_{X^{(i)}}(t_0) >0$,
%we have $\dimH(\gr_{X_{\bullet}(\omega)}([a,b]))\leq \min\left\{N/\alpha ; N + d(1-\alpha)\right\}$.

%Therefore, taking two sequences $(a_n)_{n\in\N}$ and $(b_n)_{n\in\N}$ such that $\forall n\in\N$, $a_n<t_0<b_n$ and converging to $t_0$, we get
%\begin{align*}
%\forall t_0\in\mathcal{A},\quad 
%\lim_{n\rightarrow \infty}&\dimH(\gr_{X_{\bullet}(\omega)}([a_n,b_n])) \\
%&\leq \min\left\{N/\liminf_{u\rightarrow t}\alphalocbb_{X^{(i)}}(u) ; N + d(1-\liminf_{u\rightarrow t}\alphalocbb_{X^{(i)}}(u))\right\}.
%\end{align*}
%By monotony of the Hausdorff dimension, the result follows.
%\end{proof}

\

\subsection{Lower bound for the Hausdorff dimension}\label{sec:low}

Frostman's Theorem constitutes the key argument to prove the lower bound for the Hausdorff dimensions. We recall the basic notions of potential theory, which are used along the proofs of this section.
For any Borel set $E\subseteq\R^d$, the $\beta$-dimensional energy of a probability measure $\mu$ on $E$ is defined by
\begin{equation*}
I_{\beta}(\mu) = \int_{E\times E} \|x-y\|^{-\beta}\ \mu(dx)\ \mu(dy).
\end{equation*}
Then, the $\beta$-dimensional Bessel-Riesz capacity of $E$ is defined as
$$ C_{\beta}(E) = \sup\left( \frac{1}{I_{\beta}(\mu)};\; \mu\textrm{ probability measure on }E \right). $$
According to Frostman's Theorem, the Hausdorff dimension of $E$ is obtained from the capacity of $E$ by the expression
\begin{align*}
\dimH E = \sup\left(\beta: C_{\beta}(E)>0\right)
=\inf\left(\beta: C_{\beta}(E)=0\right).
\end{align*}

Consequently, if $I_{\beta}(\mu) < +\infty$ for some probability measure (or some mass distribution) $\mu$ on $E$, then $\dimH E \geq \beta$.

\begin{proposition}\label{propmindimH}
Let $X=\{X_t;\;t\in\R^N_+\}$ be a multiparameter Gaussian field in $\R^d$ and $\alphainfbb_{X^{(i)}}(t_0)$ the deterministic local sub-exponent of $X^{(i)}$ at $t_0\in\R_+^N$.

Then, almost surely 
\begin{align*}
\lim_{\rho\rightarrow 0}\dimH(\gr_X(B(t_0,\rho))) \geq 
\left\{ \begin{array}{l l}
N/\alphainfbb_{X^{(i)}}(t_0) 
& \textrm{if }N\leq d\ \alphainfbb_{X^{(i)}}(t_0) ; \\
N + d(1-\alphainfbb_{X^{(i)}}(t_0)) 
& \textrm{if }N > d\ \alphainfbb_{X^{(i)}}(t_0) ;
\end{array} \right.
\end{align*}
and
\begin{align*}
\lim_{\rho\rightarrow 0}\dimH(\rg_X(B(t_0,\rho))) \geq 
\left\{ \begin{array}{l l}
N/\alphainfbb_{X^{(i)}}(t_0) 
& \textrm{if }N\leq d\ \alphainfbb_{X^{(i)}}(t_0) ; \\
d
& \textrm{if }N > d\ \alphainfbb_{X^{(i)}}(t_0).
\end{array} \right.
\end{align*}

\end{proposition}

\begin{proof}
Following the Adler's proof for the lower bound in the case of processes with stationary increments, we distinguish the two cases: $N\leq d\ \alphainfbb_{X^{(i)}}(t_0)$ and $N > d\ \alphainfbb_{X^{(i)}}(t_0)$. 

\begin{itemize}
\item Assume that $N\leq d\ \alphainfbb_{X^{(i)}}(t_0)$. 
In that case, we prove that almost surely,
\begin{align}\label{eqmindimH1}
\lim_{\rho\rightarrow 0}\dimH(\gr_X(B(t_0,\rho)) 
\geq \lim_{\rho\rightarrow 0}\dimH(\rg_X(B(t_0,\rho)) 
\geq \frac{N}{\alphainfbb_{X^{(i)}}(t_0)}.
\end{align}
For any $\epsilon>0$, we consider any $\beta<N/(\alphainfbb_{X^{(i)}}(t_0)+\epsilon) \leq d$ and we aim at showing that the $\beta$-dimensional capacity $C_\beta(\rg_X(B(t_0,\rho)))$ is positive almost surely for all $\rho>0$. 

\noindent With this intention, for $E=\rg_X(B(t_0,\rho))=X(B(t_0,\rho))$, we consider the $\beta$-dimensional energy $I_{\beta}(\mu)$ of the mass distribution $\mu = \lambda|_{B(t_0,\rho)} \circ X^{-1}$ of $E$, where $\lambda|_{B(t_0,\rho)}$ denotes the restriction of the Lebesgue measure to $B(t_0,\rho)$. As mentioned above (see also Theorem B in \cite{Taylor}), a sufficient condition for the capacity to be positive is that, almost surely
\begin{align}\label{eqcapa}
\int_{E\times E}\|x-y\|^{-\beta}\ \mu(dx)\ \mu(dy) =
\int_{B(t_0,\rho)\times B(t_0,\rho)}\|X_t-X_s\|^{-\beta}\ ds\ dt
< +\infty.
\end{align}
Since the $X^{(i)}$ are independent and have the same distribution, we compute for all $s,t\in\R^N_+$,
\begin{align*}
\EE\left[\|X_t-X_s\|^{-\beta}\right]=\frac{1}{[2\pi \sigma^2(s,t)]^{d/2}}
\int_{\R^d} \|x\|^{-\beta} \exp\left(-\frac{\|x\|^2}{2\ \sigma^2(s,t)}\right)\ dx,
\end{align*}
where $\sigma^2(s,t) = \EE[| X^{(i)}_t-X^{(i)}_s |^2]$ is independent of $1\leq i\leq d$.\\
Let us consider the change of variables $(\R_+\setminus\{0\}, \SS^{d-1})\rightarrow\R^d\setminus\{0\}$ defined by $(r,u)\mapsto r.u$, where $\SS^{d-1}$ denotes the unit hypersphere of $\R^d$. The previous expression becomes
\begin{align*}
\EE\left[\|X_t-X_s\|^{-\beta}\right] &= \frac{K_1}{[2\pi \sigma^2(s,t)]^{d/2}}
\int_{\R_+} r^{d-1-\beta} \exp\left(-\frac{r^2}{2\ \sigma^2(s,t)}\right)\ dr \\
&= K_1\ (\sigma(s,t))^{-\beta}\int_{\R_+} z^{d-1-\beta}\exp\left(-\frac{1}{2}z^2\right)\ dz,
\end{align*}
where $K_1$ is a positive constant and using the change of variables $r=\sigma(s,t)\ z$.\\
Since the integral is finite when $\beta<d$, we get 
\begin{equation}\label{eqCovInc_beta}
\forall s,t\in\R_+^N,\quad
\EE\left[\|X_t-X_s\|^{-\beta}\right]\leq K_2\ (\sigma(s,t))^{-\beta},
\end{equation}
for some positive constant $K_2$.\\
By Tonelli's theorem and Lemma \ref{lemcovinc}, this inequality implies the existence of $\rho_0>0$ such that for all $\rho\in (0,\rho_0]$, 
\begin{align*}
\EE\left[\int_{B(t_0,\rho)\times B(t_0,\rho)}\|X_t-X_s\|^{-\beta}\ dt\ ds\right]& \\
\leq \int_{B(t_0,\rho)\times B(t_0,\rho)}K_2\ &\|t-s\|^{-\beta(\alphainfbb_{X^{(i)}}(t_0) + \epsilon)}\ dt\ ds < +\infty
\end{align*}
because $\beta(\alphainfbb_{X^{(i)}}(t_0) + \epsilon)< N$.
Thus (\ref{eqcapa}) holds and for all $\rho\in (0,\rho_0]$,
\begin{align*}
\dimH(\rg_X(B(t_0,\rho))) 
\geq \frac{N}{\alphainfbb_{X^{(i)}}(t_0)+\epsilon} \qquad\textrm{a.s.}
\end{align*}
Taking $\rho,\epsilon\in\Q_+$, this yields to 
\begin{align*}
\lim_{\rho\rightarrow 0}\dimH(\rg_X(B(t_0,\rho))) 
\geq \frac{N}{\alphainfbb_{X^{(i)}}(t_0)} \qquad\textrm{a.s.},
\end{align*}
which proves (\ref{eqmindimH1}).

\

\item Assume $N>d\ \alphainfbb_{X^{(i)}}(t_0)$.
We use the previous method to prove that almost surely
\begin{equation}\label{eqmindimHrg2}
\lim_{\rho\rightarrow 0}\dimH(\rg_X(B(t_0,\rho)))\geq d.
\end{equation}
For any $\epsilon>0$ such that $d<N/(\alphainfbb_{X^{(i)}}(t_0)+\epsilon)$, consider any real $\beta$ such that $\beta<d$. 
As previously, we show that equation (\ref{eqcapa}) is verified, which implies that the $\beta$-dimensional capacity $C_\beta(\rg_X(B(t_0,\rho)))$ is positive almost surely for all $\rho>0$.

\noindent Since $\beta<d$, equation (\ref{eqCovInc_beta}) still holds.
As in the previous case, the inequality $\beta(\alphainfbb_{X^{(i)}}(t_0) + \epsilon)< N$ implies (\ref{eqcapa}) for $\rho$ small enough and then
\begin{align*}
\dimH(\rg_X(B(t_0,\rho))) \geq d \qquad\textrm{a.s.}
\end{align*}
Taking $\rho\in\Q_+$, the inequality (\ref{eqmindimHrg2}) follows.

\

\item Assume $N>d\ \alphainfbb_{X^{(i)}}(t_0)$.
To prove the lower bound for the Hausdorff dimension of the graph,  
\begin{equation}\label{eqmindimH2}
\lim_{\rho\rightarrow 0}\dimH(\gr_X(B(t_0,\rho)))\geq N + d(1-\alphainfbb_{X^{(i)}}(t_0)) \qquad\textrm{a.s.},
\end{equation}
we use the same arguments of potential theory than for the range.

\noindent For any $\epsilon>0$, consider any real $\beta$ such that $d<\beta<N+d(1-\alphainfbb_{X^{(i)}}(t_0)-\epsilon)$. 
In order to prove that the $\beta$-dimensional capacity $C_\beta(\gr_X(B(t_0,\rho)))$ is positive almost surely for all $\rho>0$, it is sufficient to show that 
\begin{equation}\label{eqcapa2}
\int_{B(t_0,\rho)\times B(t_0,\rho)} \|(t,X_t)-(s,X_s)\|^{-\beta}\ ds\ dt<+\infty 
\qquad\textrm{a.s.}
\end{equation}

\noindent
Since the components $X^{(i)}$ ($1\leq i\leq d$) of $X$ are i.i.d., we compute
\begin{align*}
&\EE\left[(\|X_t-X_s\|^2+\|t-s\|^2)^{-\beta/2}\right] \\
&\qquad =\frac{1}{[2\pi \sigma^2(s,t)]^{d/2}}
\int_{\R^d}\left(\|x\|^2+\|t-s\|^2\right)^{-\beta/2}
\exp\left(-\frac{\|x\|^2}{2\ \sigma^2(s,t)}\right)\ dx.
\end{align*}
As in the previous case, by using the hyperspherical change of variables $(r,u)\in\R_+\times\SS^{d-1}$ and then $r=\sigma(s,t)\ z$, we get
\begin{align*}
&\EE\left[(\|X_t-X_s\|^2 + \|t-s\|^2)^{-\beta/2}\right]  \\
&\qquad = K_3 \int_{\R_+} \left(z^2\sigma^2(s,t)+\|t-s\|^2\right)^{-\beta/2}\ z^{d-1}\ e^{-\frac{1}{2}z^2}\ dz \\
&\qquad = K_3\ \sigma(s,t)^{-\beta} \int_{\R_+} \left(z^2+\frac{\|t-s\|^2}{\sigma^2(s,t)}\right)^{-\beta/2}\ z^{d-1}\ e^{-\frac{1}{2}z^2}\ dz,
\end{align*}
where $K_3$ is a positive constant. Then, since $\beta>d$, the following inequality holds
\begin{align*}
&\EE\left[(\|X_t-X_s\|^2 + \|t-s\|^2)^{-\beta/2}\right]  \\
&\qquad \leq \frac{2^{-\beta/2}\ K_3}{\sigma(s,t)^{\beta}} \left[ \int_0^\frac{\|t-s\|}{\sigma(s,t)} 
\left( \frac{\|t-s\|}{\sigma(s,t)} \right)^{-\beta} z^{d-1}\ dz
+ \int_\frac{\|t-s\|}{\sigma(s,t)}^\infty z^{d-1-\beta}\ dz \right] \\
&\qquad \leq \frac{K_4}{\sigma(s,t)^{\beta}}\ \left( \frac{\|t-s\|}{\sigma(s,t)} \right)^{d-\beta} 
\leq K_4\ \frac{\|t-s\|^{d-\beta}}{\sigma(s,t)^d}.
\end{align*}

\noindent By Tonelli's Theorem and Lemma \ref{lemcovinc}, this inequality implies the existence of $\rho_0>0$ such that for all $\rho\in (0,\rho_0]$, 
\begin{align*}
&\EE\left[ \int_{B(t_0,\rho)\times B(t_0,\rho)} \|(t,X_t)-(s,X_s)\|^{-\beta}\ dt\ ds \right]\\
&\qquad\qquad\leq \int_{B(t_0,\rho)\times B(t_0,\rho)}  K_4\ 
\frac{\|t-s\|^{d-\beta}}{\sigma(s,t)^d}\ ds\ dt \\
&\qquad\qquad\leq \int_{B(t_0,\rho)\times B(t_0,\rho)}  K_4\ 
\|t-s\|^{-\beta+d(1-\alphainfbb_{X^{(i)}}(t_0)-\epsilon)}\ ds\ dt
<+\infty,
\end{align*}
because $\beta<N+d(1-\alphainfbb_{X^{(i)}}(t_0)-\epsilon)$. 
Thus (\ref{eqcapa2}) holds and for all $\rho\in (0,\rho_0]$,
\begin{equation*}
\dimH(\gr_X(B(t_0,\rho)) \geq N+d(1-\alphainfbb_{X^{(i)}}(t_0)-\epsilon)
\qquad\textrm{a.s.}
\end{equation*}
Taking $\rho,\epsilon\in\Q_+$, this yields to
\begin{equation*}
\lim_{\rho\rightarrow 0}\dimH(\gr_X(B(t_0,\rho)) \geq
N+d(1-\alphainfbb_{X^{(i)}}(t_0)) \qquad\textrm{a.s.}, 
\end{equation*}
which proves (\ref{eqmindimH2}).
\end{itemize}
\end{proof}

We now investigate uniform extensions of Proposition \ref{propmindimH}.

\begin{corollary}\label{cormindimHunif1}
Let $X=\{X_t;\;t\in\R^N_+\}$ be a multiparameter Gaussian field in $\R^d$ and $\alphainfbb_{X^{(i)}}(t)$ the deterministic local sub-exponent of $X^{(i)}$ at any $t\in\R_+^N$.

Assume that for some open subset $I \subset\R^N_+$, we have 
$\underline{\alpha} = \inf_{t\in I} \alphainfbb_{X^{(i)}}(t) > 0$.

Then, with probability one,
\begin{align*}
\dimH(\gr_X(I)) \geq 
\left\{ \begin{array}{l l}
N/\underline{\alpha} & \textrm{if } N \leq d\ \underline{\alpha} ; \\
N + d(1-\underline{\alpha}) & \textrm{if } N > d\ \underline{\alpha} ; 
\end{array} \right.
\end{align*}
and
\begin{align*}
\dimH(\rg_X(I)) \geq
\left\{ \begin{array}{l l}
N/\underline{\alpha} & \textrm{if } N \leq d\ \underline{\alpha} ; \\
d & \textrm{if } N > d\ \underline{\alpha}.
\end{array} \right.
\end{align*}
\end{corollary}

\begin{proof}
For any open subset $I\subset\R^N_+$, we first prove that for all $\omega$, the Hausdorff dimension of the graph of $X_{\bullet}(\omega):t\mapsto X_t(\omega)$ satisfies
\begin{align}\label{ineqdimHboules}
\dimH(\gr_{X_{\bullet}(\omega)}(I)) \geq \sup_{t_0\in I} \lim_{\rho\rightarrow 0}\dimH(\gr_{X_{\bullet}(\omega)}(B(t_0,\rho))).
\end{align}

Since $I$ is an open subset of $\R^N_+$, for all $t_0\in I$, there exists $\rho>0$ such that $B(t_0,\rho)\subset I$.
This leads to $\dimH(\gr_{X_{\bullet}(\omega)}(B(t_0,\rho))) \leq \dimH(\gr_{X_{\bullet}(\omega)}(I))$ and then
$$ \dimH(\gr_{X_{\bullet}(\omega)}(I)) \geq \lim_{\rho\rightarrow 0}\dimH(\gr_{X_{\bullet}(\omega)}(B(t_0,\rho))), $$
since $\rho\mapsto \dimH(\gr_{X_{\bullet}(\omega)}(B(t_0,\rho))) $ is decreasing.
Then (\ref{ineqdimHboules}) follows.

In the same way, we prove that for all $\omega$,
\begin{align}\label{ineqdimHboulesRg}
\dimH(\rg_{X_{\bullet}(\omega)}(I)) \geq \sup_{t_0\in I} \lim_{\rho\rightarrow 0}\dimH(\rg_{X_{\bullet}(\omega)}(B(t_0,\rho))).
\end{align}

Following the proof of Proposition \ref{propmindimH}, we distinguish the two cases: $N \leq d\ \underline{\alpha}$ and $N > d\ \underline{\alpha}$ with $\underline{\alpha} = \inf_{t\in I} \alphainfbb_{X^{(i)}}(t)$.

\begin{itemize}
\item Assume that $N \leq d\ \underline{\alpha}$. In that case, for all $t_0\in I$, we have $N \leq d\ \alphainfbb_{X^{(i)}}(t_0)$.
Equations (\ref{eqmindimH1}), (\ref{ineqdimHboules}) and (\ref{ineqdimHboulesRg}) imply almost surely
\begin{equation*}
\dimH(\gr_{X_{\bullet}(\omega)}(I)) \geq \dimH(\rg_{X_{\bullet}(\omega)}(I)) 
\geq \frac{N}{\underline{\alpha}}.
\end{equation*}

\item Assume that $N > d\ \underline{\alpha}$. By definition of $\underline{\alpha}$, for all $\epsilon>0$ with $N>d\ (\underline{\alpha} + \epsilon)$, there exists $t_0\in I$ such that
$$ \underline{\alpha} < \alphainfbb_{X^{(i)}}(t_0) < \underline{\alpha} + \epsilon. $$
Then, we have $N > d\ \alphainfbb_{X^{(i)}}(t_0)$. In the proof of Proposition \ref{propmindimH}, we proved that this implies almost surely
\begin{equation*}
\lim_{\rho\rightarrow 0}\dimH(\rg_X(B(t_0,\rho)))\geq d
\end{equation*}
and
\begin{align*}
\lim_{\rho\rightarrow 0}\dimH(\gr_X(B(t_0,\rho))) 
&\geq N + d(1-\alphainfbb_{X^{(i)}}(t_0)) \\
&\geq N + d(1-\underline{\alpha} - \epsilon)
\end{align*}
for all $\epsilon \in\Q_+$ with $N>d\ (\underline{\alpha} + \epsilon)$. 
Then almost surely,
\begin{equation*}
\sup_{t_0\in I}\lim_{\rho\rightarrow 0}\dimH(\rg_X(B(t_0,\rho)))\geq d
\end{equation*}
and
\begin{align*}
\sup_{t_0\in I}\lim_{\rho\rightarrow 0}\dimH(\gr_X(B(t_0,\rho))) 
\geq N + d(1-\underline{\alpha}).
\end{align*}
\end{itemize}
\end{proof}

\begin{corollary}\label{cormindimHunif2}
Let $X=\{X_t;\;t\in\R^N_+\}$ be a multiparameter Gaussian field in $\R^d$ and $\alphainfbb_{X^{(i)}}(t)$ the deterministic local sub-exponent of $X^{(i)}$ at any $t\in\R_+^N$.\\
Set $\underline{\mathcal{A}} = \{ t\in\R_+^N: \liminf_{u\rightarrow t}\alphainfbb_{X^{(i)}}(u)>0\}$. 

Then, with probability one, for all $t_0\in\underline{\mathcal{A}}$,
\begin{align*}
\lim_{\rho\rightarrow 0}\dimH(\gr_X(B(t_0,\rho))) \geq 
\left\{ \begin{array}{l l} 
\displaystyle
N/\liminf_{t\rightarrow t_0}\alphainfbb_{X^{(i)}}(t) 
& \displaystyle\textrm{if }N\leq d\ \liminf_{t\rightarrow t_0}\alphainfbb_{X^{(i)}}(t);\\ 
\displaystyle 
N + d\left(1-\liminf_{t\rightarrow t_0}\alphainfbb_{X^{(i)}}(t) \right)
& \displaystyle\textrm{if }N>d\ \liminf_{t\rightarrow t_0}\alphainfbb_{X^{(i)}}(t) ; 
\end{array} \right.
\end{align*}
and
\begin{align*}
\lim_{\rho\rightarrow 0}\dimH(\rg_X(B(t_0,\rho))) \geq 
\left\{ \begin{array}{l l}
\displaystyle
N/\liminf_{t\rightarrow t_0}\alphainfbb_{X^{(i)}}(t)
& \displaystyle\textrm{if } N \leq d\ \liminf_{t\rightarrow t_0} \alphainfbb_{X^{(i)}}(t) ; \\
d & \displaystyle\textrm{if } N>d\ \liminf_{t\rightarrow t_0} \alphainfbb_{X^{(i)}}(t). 
\end{array} \right.
\end{align*}
\end{corollary}

%{\bf This result comes from Corollary \ref{cormindimHunif1} in the same way as Corollary \ref{cormajdimHunif2} is proved using Corollary \ref{cormajdimHunif1}. May be we could skip the proof. Except if Corollary \ref{cormajdimHunif2} is skipped.}

\begin{proof}
Corollary \ref{cormindimHunif1} implies the existence of $\Omega^*\in\mathcal{F}$ with $\PP(\Omega^*)=1$ such that: For all $\omega\in\Omega^*$ and all $a,b\in\Q^N_+$ with $a\prec b$, such that $\underline{\alpha} = \inf_{t\in (a,b)} \alphainfbb_{X^{(i)}}(t) > 0$, we have 
$\dimH(\gr_{X_{\bullet}(\omega)}((a,b))) 
\geq N/\underline{\alpha}$ if $N\leq d\ \underline{\alpha}$ and $\geq N + d(1-\underline{\alpha})$ if $N> d\ \underline{\alpha}$ and 
$\dimH(\rg_{X_{\bullet}(\omega)}((a,b))) 
\geq N/\underline{\alpha}$ if $N\leq d\ \underline{\alpha}$ and $\geq d$ if $N> d\ \underline{\alpha}$.

Therefore, taking two sequences $(a_n)_{n\in\N}$ and $(b_n)_{n\in\N}$ such that $\forall n\in\N$, $a_n<t_0<b_n$ and converging to $t_0$, we get
\begin{align*}
\lim_{n\rightarrow\infty}\dimH(\gr_{X_{\bullet}(\omega)}((a_n,b_n))) \geq 
\left\{ \begin{array}{l l} 
\displaystyle
N/\liminf_{t\rightarrow t_0}\alphainfbb_{X^{(i)}}(t) 
& \displaystyle\textrm{if }N\leq d\ \liminf_{t\rightarrow t_0}\alphainfbb_{X^{(i)}}(t);\\ 
\displaystyle 
N + d(1-\liminf_{t\rightarrow t_0}\alphainfbb_{X^{(i)}}(t) )
& \displaystyle\textrm{if }N>d\ \liminf_{t\rightarrow t_0}\alphainfbb_{X^{(i)}}(t) ; 
\end{array} \right.
\end{align*}
and
\begin{align*}
\lim_{n\rightarrow\infty}\dimH(\rg_{X_{\bullet}(\omega)}((a_n,b_n))) \geq 
\left\{ \begin{array}{l l}
\displaystyle
N/\liminf_{t\rightarrow t_0}\alphainfbb_{X^{(i)}}(t)
& \displaystyle\textrm{if } N \leq d\ \liminf_{t\rightarrow t_0}\alphainfbb_{X^{(i)}}(t) ; \\
d & \displaystyle\textrm{if } N>d\ \liminf_{t\rightarrow t_0}\alphainfbb_{X^{(i)}}(t). 
\end{array} \right.
\end{align*}
By monotony of the Hausdorff dimension, the result follows.
\end{proof}

\medskip

%%%%%%%%%%%%%%%%%%% Applications %%%%%%%%%%%%%%%%%%

\section{Applications}\label{sec:app}

In this section, we apply the main results to Gaussian processes whose fine regularity is not completely known: the multiparameter fractional Brownian motion, the multifractional Brownian motion with a regularity function lower than its own regularity and the generalized Weierstrass function.

\subsection{Multiparameter fractional Brownian motion}\label{sec:mpfbm}

The multiparameter fractional Brownian motion (MpfBm) $\sifbm^H=\{ \sifbm^H_t;\; t\in\R_+^N\}$ of index $H\in (0,1/2]$ is defined as a particular case of set-indexed fractional Brownian motion (see \cite{sifBm, MpfBm}), where the indexing collection is $\mathcal{A}=\{ [0,t];\; t\in\R_+^N\} \cup\{\emptyset\}$.
It is characterized as a real-valued mean-zero Gaussian process with covariance function
\begin{align*}
\forall s,t\in\R_+^N,\quad
\EE[\sifbm^H_s \sifbm^H_t] = \frac{1}{2} \left[ m([0,s])^{2H} + m([0,t])^{2H} - m([0,s]\bigtriangleup [0,t])^{2H} \right],
\end{align*}
where $m$ denotes a Radon measure in $\R^N_+$.

In the specific case where $N=2$ and $m$ is the Lebesgue measure of $\R^2_+$, the covariance structure of the MpfBm is 
\begin{align*}
\forall s,t\in\R_+^2,\quad
\EE[\sifbm^H_s \sifbm^H_t] = \frac{1}{2} \left[ (s_1s_2)^{2H}+(t_1t_2)^{2H}-(s_1s_2+t_1t_2-2(s_1\wedge t_1)(s_2\wedge t_2))^{2H} \right].
\end{align*}

Then, its incremental variance is
\begin{align}\label{eqcovincMpfBm}
\forall s,t\in \R_+^2,\quad 
\EE\left[|\sifbm^H_t-\sifbm^H_s|^2\right]=(s_1 s_2 + t_1 t_2 - 2 (s_1\wedge t_1)(s_2\wedge t_2))^{2H}.
\end{align}

The stationarity of the increments of the multiparameter fractional Brownian motion are studied in \cite{MpfBm}. Among all the various definitions of the stationarity property for a multiparameter process, the MpfBm does not satisfy the increment stationarity assumption of \cite{Adler77}. Indeed, (\ref{eqcovincMpfBm}) shows that $\EE\left[|\sifbm^H_t-\sifbm^H_s|^2\right]$ does not only depend on $t-s$.
Since the Hausdorff dimension of its graph does not come directly from \cite{Adler77},
we use the generic results of Section \ref{sec:main}.

\begin{lemma}\label{lemdist}
If $m$ is the Lebesgue measure of $\R^N$, for any $a\prec b$ in 
$\R^N_{+}\setminus\left\{0\right\}$, there exists two positive constants
$m_{a,b}$ and $M_{a,b}$ such that
\begin{equation*}
\forall s,t\in [a,b];\quad
m_{a,b}\ d_1(s,t)\leq m([0,s]\bigtriangleup [0,t])
\leq M_{a,b}\ d_{\infty}(s,t)
\end{equation*}
where $d_1$ and $d_{\infty}$ are the usual distances of $\R^N$ defined by
\begin{align*}
d_1:(s,t)&\mapsto\|t-s\|_1=\sum_{i=1}^N |t_i-s_i| \\
d_{\infty}:(s,t)&\mapsto\|t-s\|_{\infty}=\max_{1\leq i\leq N} |t_i-s_i|.
\end{align*}

\end{lemma}

\begin{proof}
For all $s,t\in [a,b]$, we write
\begin{align*}
[0,s] \bigtriangleup [0,t]=\left([0,s]\setminus [0,t]\right) \cup 
\left([0,t]\setminus [0,s]\right).
\end{align*}
Suppose that for all $i\in I\subset\left\{1,\dots,N\right\}$, $s_i>t_i$, 
and that for all $i\in\left\{1,\dots,N\right\}\setminus I$, $s_i\leq t_i$.
%For sake of simplicity of notations, we assume that $I= \{1,2,\dots,I\}$ and $I^c=\{I+1,\dots,N\}$.
For any subset $J$ of $\{1,\dots,N\}$, we denote by $\prod_{i\in J}[0,s_i]$ the cartesian product of $[0,s_i]$ for $i\in J$.

\noindent We have
\begin{align*}
[0,s] &= \prod_{i\notin I}[0,s_i] \times \prod_{i\in I}\left([0,t_i]\cup [t_i,s_i]\right) \\
&= \left( \prod_{i\notin I}[0,s_i] \times \prod_{i\in I}[0,t_i] \right) \cup
\bigcup_{J\subsetneq I}\left( \prod_{i\notin I}[0,s_i] \times \prod_{i\in J}[0,t_i]
\times \prod_{i\in I\setminus J}[t_i,s_i] \right),
\end{align*}
and then
\begin{align*}
[0,s]\setminus [0,t] &= \bigcup_{J\subsetneq I}\left( \prod_{i\notin I}[0,s_i] \times \prod_{i\in J}[0,t_i] \times \prod_{i\in I\setminus J}[t_i,s_i] \right) \\
&= \left\{ x\in [0,s]:\;\exists i\in I;\; t_i<x_i\leq s_i \right\}.
\end{align*}
We deduce 
\begin{align*}
m([0,s]\setminus [0,t]) = \prod_{i\notin I} |s_i|\  \sum_{J\subsetneq I} \left( \prod_{i\in J} |t_i|
\prod_{i\in I\setminus J} |t_i-s_i| \right).
\end{align*}
In the same way, we get
\begin{align*}
m([0,t]\setminus [0,s]) = \prod_{i\in I} |s_i|\  \sum_{J\subsetneq I^c} \left( 
\prod_{i\in J} |t_i| \prod_{i\in I^c\setminus J} |t_i-s_i| \right).
\end{align*}

\noindent 
For all $1\leq i\leq N$, we have $|a| \leq |s_i| \leq |b|$ and $|a| \leq |t_i| \leq |b|$.
Then, 
\begin{align*}
m([0,s] &\bigtriangleup [0,t]) \\
&\leq |b|^{\# I^c} \sum_{J\subsetneq I} |b|^{\# J} 
d_{\infty}(s,t)^{\#(I\setminus J)} + |b|^{\# I} \sum_{J\subsetneq I^c} |b|^{\# J} 
d_{\infty}(s,t)^{\#(I^c\setminus J)} \\
&\leq d_{\infty}(s,t)\ \underbrace{\left[ |b|^{\# I^c} \sum_{J\subsetneq I} |b|^{\# J} 
d_{\infty}(s,t)^{\#(I\setminus J)-1} + |b|^{\# I} \sum_{J\subsetneq I^c} |b|^{\# J} 
d_{\infty}(s,t)^{\#(I^c\setminus J)-1}\right]}_{\textrm{bounded in }[a,b]}\\
&\leq M_{a,b}\ d_{\infty}(s,t).
\end{align*}
For the lower bound, we write
\begin{align*}
m([0,s] \bigtriangleup [0,t]) 
\geq |a|^{\# I^c} \sum_{J\subsetneq I} |a|^{\# J} 
\prod_{i\in I\setminus J} |t_i-s_i| + |a|^{\# I} \sum_{J\subsetneq I^c} |a|^{\# J} 
\prod_{i\in I^c\setminus J} |t_i-s_i| 
\end{align*}
Let $m_a$ be the minimum of $|a|^k$ for $1\leq k\leq N$. We get
\begin{align}\label{eqminda}
m([0,s] \bigtriangleup [0,t])
\geq m_a^2 \sum_{J\subsetneq I}  \prod_{i\in I\setminus J} |t_i-s_i| 
+ m_a^2 \sum_{J\subsetneq I^c} \prod_{i\in I^c\setminus J} |t_i-s_i|.
\end{align}

\noindent Let us remark that
\begin{align*}
\sum_{J\subsetneq I}  \prod_{i\in I\setminus J} |t_i-s_i| 
= \prod_{i\in I} \left(1+|t_i-s_i|\right) - 1.
\end{align*}
Using the expansion
\begin{align*}
\log\prod_{i\in I} \left(1+|t_i-s_i|\right) = \sum_{i\in I}\log\left(1+|t_i-s_i|\right)
=\sum_{i\in I} |t_i-s_i| + o(|t_i-s_i|^2),
\end{align*}
which implies
\begin{align*}
\prod_{i\in I} \left(1+|t_i-s_i|\right) = 1+\sum_{i\in I} |t_i-s_i| + o(|t_i-s_i|^2),
\end{align*}
the inequality (\ref{eqminda}) becomes
\begin{align*}
m([0,s] \bigtriangleup [0,t])
\geq m_a^2 \sum_{1\leq i\leq N} |t_i-s_i| + o(\|t-s\|_{\infty}).
\end{align*}
The result follows.
\end{proof}

\medskip

\begin{lemma}\label{lemMpfBmexp}
Let $\sifbm^H=\{ \sifbm^H_t;\; t\in\R_+^N\}$ be a multiparameter fractional Brownian motion with index $H\in (0, 1/2]$. The deterministic local H\"older exponent and deterministic local sub-exponent of $\sifbm^H$ at any $t_0\in\R_+^N$ is given by
$\alphalocbb_X(t_0)=\alphainfbb_X(t_0)=H$.
\end{lemma}

\begin{proof}
We prove that $\alphalocbb_X(t_0)\geq H$ and $\alphainfbb_X(t_0)\leq H$.
The result will follow from $\alphalocbb_X(t_0)\leq\alphainfbb_X(t_0)$.

Since for all $s,t\in \R^N_+$,
\begin{align*}
\frac{\EE\left[|\sifbm^H_t - \sifbm^H_t|^2\right]}{\|t-s\|^{2H}} = 
\left( \frac{m([0,s] \bigtriangleup [0,t])}{d_2(s,t)} \right)^{2H},
\end{align*}
Lemma \ref{lemdist} implies that for all $s,t$ in any interval $[a,b]$,
\begin{align}\label{ineqMpfBm}
M_1 \left( \frac{d_1(s,t)}{d_2(s,t)} \right)^{2H} \leq
\frac{\EE\left[|\sifbm^H_t - \sifbm^H_t|^2\right]}{\|t-s\|^{2H}} 
\leq M_2 \left( \frac{d_{\infty}(s,t)}{d_2(s,t)} \right)^{2H},
\end{align}
for some positive constants $M_1$ and $M_2$.

\noindent Since the distances $d_1, d_2$ and $d_{\infty}$ are equivalent, the inequality (\ref{ineqMpfBm}) implies that the quantity $\EE\left[|\sifbm^H_t - \sifbm^H_t|^2\right]/\|t-s\|^{2H}$ is bounded on any interval $[a,b]$. 
Consequently, for all $t_0\in\R^N_+$, $\alphalocbb_X(t_0)\geq H$ and 
$\alphainfbb_X(t_0)\leq H$, by definition of the deterministic local H\"older exponent and the deterministic local sub-exponent.
\end{proof}

\

A direct consequence from Lemma \ref{lemMpfBmexp} is the local regularity of the sample paths of the multiparameter fractional Brownian motion. In \cite{2ml}, Corollary 3.15 states that for any Gaussian process $X$ such that the function $t\mapsto\alphalocbb_X(t)$ is continuous and positive, the local H\"older exponents satisfy with probability one: $\widetilde{\alphar}_X(t)=\alphalocbb_X(t)$ for all $t\in\R_+^N$. Since the deterministic local H\"older exponents of the MpfBm are constant and positive, the following result comes directly.

\begin{corollary}
The local H\"older exponent of the multiparameter fractional Brownian motion $\sifbm^H=\{ \sifbm^H_t;\; t\in\R_+^N\}$ (with $1<H\leq 1/2$) satisfies with probability one, $\widetilde{\alphar}_{\sifbm^H}(t_0)=H$ for all $t_0\in\R^N_+$.
\end{corollary}

\medskip

As an application of Theorem \ref{thmaincompact}, the property of constant local regularity of the multiparameter fractional Brownian motion yields to sharp results about the Hausdorff dimensions of its graph and its range. 

\begin{proposition}\label{prop:mpfbm}
Let $X=\{ X_t;\; t\in\R_+^N\}$ be a multiparameter fractional Brownian field with index $H\in (0, 1/2]$, i.e. whose coordinate processes $X^{(1)},\dots,X^{(d)}$ are i.i.d. multiparameter fractional Brownian motions with index $H$. \\
With probability one, the Hausdorff dimensions of the graph and the range of the sample paths of $X$ are
\begin{align*}
\forall I=(a,b)\subset\R_+^N,\quad
\dimH(\gr_{X}(I)) &= \min\{ N / H; N + d(1 - H) \}, \\
\dimH(\rg_{X}(I)) &= \min\{ N / H; d \}.
\end{align*}
\end{proposition}

\medskip

\begin{corollary}\label{cor:mpfbm}
Let $\sifbm^H=\{ \sifbm^H_t;\; t\in\R_+^N\}$ be a multiparameter fractional Brownian motion with index $H\in (0, 1/2]$. 
With probability one, the Hausdorff dimensions of the graph and the range of the sample paths of $\sifbm^H$ are
\begin{align*}
\forall I=(a,b)\subset\R_+^N,\quad
\dimH(\gr_{\sifbm^H}(I)) &= N + 1 - H, \\
\dimH(\rg_{\sifbm^H}(I)) &= 1.
\end{align*}
\end{corollary}

\medskip

Proposition \ref{prop:mpfbm} and Corollary \ref{cor:mpfbm} should be compared to Theorem 1.3 of \cite{AyXiao} which states the Hausdorff dimensions of the range and the graph of the fractional Brownian sheet (result extended by Proposition 1 and Theorem 3 of \cite{WuXiao}). In particular, the Hausdorff dimensions of the sample path (range and graph) of the multiparameter fractional Brownian motion are equal to the respective quantities for the fractional Brownian sheet, when the Hurst index is the same along each axis.

\medskip

\subsection{Irregular Multifractional Brownian motion}\label{sec:mbm}

The multifractional Brownian motion (mBm) is an extension of the fractional Brownian motion, where the self-similarity index $H\in (0,1)$ is substituted with a function $H:\R_+\rightarrow (0,1)$ (see \cite{RPJLV} and \cite{BJR}). More precisely, it can be defined as a zero mean Gaussian process $\{X_t;\;t\in\R_+\}$ with
\begin{equation*}
X_t = \int_{-\infty}^t \left[(t-u)^{H(t)-1/2} - (-u)^{H(t)-1/2}\right].\mathbbm{W}(du)
+\int_0^t (t-u)^{H(t)-1/2}.\mathbbm{W}(du)
\end{equation*}
or
\begin{equation}\label{eq:mbm-harm}
X_t = \int_{\R} \frac{e^{it\xi}-1}{|\xi|^{H(t)+1/2}}.\widehat{\mathbbm{W}}(du),
\end{equation}
where $\mathbbm{W}$ is a Gaussian measure in $\R$ and $\widehat{\mathbbm{W}}$ is the Fourier transform of a Gaussian measure in $\C$.
The variety of the class of multifractional Brownian motions is described in \cite{StoevTaqqu}.

In the first definitions of the mBm, the different groups of authors used to consider the assumption: $H$ is a $\beta$-H\"older function and $H(t)<\beta$ for all $t\in\R_+$.
Under this so-called $(H_{\beta})$-assumption, the local regularity of the sample paths was described by
\begin{equation*}
\alphar_X(t_0) = \widetilde{\alphar}_X(t_0) = H(t_0) \qquad\textrm{a.s.}
\end{equation*}
where $\alphar_X(t_0)$ and $\widetilde{\alphar}_X(t_0)$ denote the pointwise and local H\"older exponents of $X$ at any $t_0\in\R_+$.
A localization of the Hausdorff dimension of the graph were also proved: For any $t_0\in\R_+$,
\begin{equation*}
\lim_{\rho\rightarrow 0}\dimH\left[ \gr_X\left( B(t_0,\rho) \right)\right] = 2-H(t_0)
\qquad\textrm{a.s.}
\end{equation*}
Let us notice that this result could not be a direct consequence of Adler's earlier work \cite{Adler77} since the multifractional Brownian motion does not have stationary increments, on the contrary to the classical fractional Brownian motion.

In \cite{EH06, 2ml}, the fine regularity of the multifractional Brownian motion has been studied in the irregular case, i.e. when the function $H$ is only assumed to be $\beta$-H\"older continuous with $\beta>0$. In this more general case, the pointwise and local H\"older exponents of $X$ at any $t_0\in\R_+$ satisfy respectively
\begin{align*}
\alphar_X(t_0) &= H(t_0) \wedge \alpha_H(t_0) \qquad\textrm{a.s.}\\
\widetilde{\alphar}_X(t_0) &= H(t_0) \wedge \widetilde{\alpha}_H(t_0) \qquad\textrm{a.s.},
\end{align*}
where
\begin{align*}
\alpha_H(t_0)&= \sup\left\{ \alpha>0: \limsup_{\rho\rightarrow 0}
\sup_{s,t\in B(t_0,\rho)} \frac{|H(t)-H(s)|}{\rho^{\alpha}} < +\infty \right\};\\
\widetilde{\alpha}_H(t_0)&= \sup\left\{ \alpha>0: \lim_{\rho\rightarrow 0}
\sup_{s,t\in B(t_0,\rho)} \frac{|H(t)-H(s)|}{|t-s|^{\alpha}} < +\infty \right\}.
\end{align*}

Roughtly speaking, when the function $H$ is irregular, it transmits its local regularity to the sample paths of the mBm.
But in that case, nothing is known about the Hausdorff dimension of the range or the graph of the process.

In this section, the main results of the paper stated in Section \ref{sec:main} are applied to derive informations on these Hausdorff dimensions, without any regularity assumptions on the function $H$.
As for Gaussian processes, we define the {\em local sub-exponent} of $H$ at $t_0\in\R_+$ by
\begin{align*}
\underline{\alpha}_H(t_0)&= \inf\left\{ \alpha>0 :
\lim_{\rho\rightarrow 0} \inf_{s,t\in B(t_0,\rho)} \frac{|H(t)-H(s)|}{|t-s|^{\alpha}}
=+\infty \right\} \\
&= \sup\left\{ \alpha>0 :
\lim_{\rho\rightarrow 0} \inf_{s,t\in B(t_0,\rho)} \frac{|H(t)-H(s)|}{|t-s|^{\alpha}}
=0 \right\}. 
\end{align*}

\medskip

\begin{proposition}\label{prop:mbm}
Let $X=\{X_t;\; t\in\R_+\}$ be the multifractional Brownian motion of integral representation (\ref{eq:mbm-harm}), with regularity function $H:\R_+\rightarrow (0,1)$ assumed to be $\beta$-H\"older-continuous with $\beta>0$. 
Let $\widetilde{\alpha}_H(t_0)$ and $\underline{\alpha}_H(t_0)$ be respectively the local H\"older exponent and sub-exponent of $H$ at $t_0\in\R_+$.

\medskip

In the three following cases, the Hausdorff dimension of the graph of the sample path of $X$ satisfies:
\begin{enumerate}[(i)]
\item If $H(t_0) < \widetilde{\alpha}_H(t_0) \leq \underline{\alpha}_H(t_0)$ for $t_0\in\R_+$, then
\begin{align*}
\lim_{\rho\rightarrow 0}\dimH(\gr_X(B(t_0,\rho))) = 2-H(t_0)  \qquad\textrm{a.s.}
\end{align*}

\item If $\widetilde{\alpha}_H(t_0) < H(t_0) \leq \underline{\alpha}_H(t_0)$ for $t_0\in\R_+$, then
\begin{align*}
2-H(t_0) \leq \lim_{\rho\rightarrow 0}\dimH(\gr_X(B(t_0,\rho))) &\leq 2-\widetilde{\alpha}_H(t_0) \qquad\textrm{a.s.}
\end{align*}

\item If $\widetilde{\alpha}_H(t_0) \leq \underline{\alpha}_H(t_0) < H(t_0)$ for $t_0\in\R_+$, then
\begin{align*}
2-\underline{\alpha}_H(t_0) \leq \lim_{\rho\rightarrow 0}\dimH(\gr_X(B(t_0,\rho))) &\leq 2-\widetilde{\alpha}_H(t_0) \qquad\textrm{a.s.} 
\end{align*}

\end{enumerate}

With probability one, the Hausdorff dimension of the range of the sample path of $X$ satisfies:
\begin{align*}
\forall t_0\in\R_+,\quad
\lim_{\rho\rightarrow 0}\dimH(\rg_X(B(t_0,\rho))) = 1.
\end{align*}

Moreover if the $(H_{\beta})$-assumption holds then, with probability one,  
\begin{align*}
\forall t_0\in\R_+,\quad
\lim_{\rho\rightarrow 0}\dimH(\gr_X(B(t_0,\rho))) = 2-H(t_0).
\end{align*}

\end{proposition}

\begin{proof}

In \cite{EH06}, an asymptotic behaviour of the incremental variance of the multifractional Brownian motion, in a neighborhood $B(t_0,\rho)$ of any $t_0\in\R_+$ as $\rho$ goes to $0$, is given by: $\forall s,t\in B(t_0,\rho)$,
\begin{align}\label{eq:asympMBM}
\EE[|X_t-X_s|^2] \sim K(t_0)\ |t-s|^{H(t)+H(s)} + L(t_0)\ [H(t)-H(s)]^2,
\end{align}
where $K(t_0)$ and $L(t_0)$ are positive constants.\\
From (\ref{eq:asympMBM}), for any $t_0\in\R_+$, for all $\alpha>0$ and for all $s,t\in B(t_0,\rho)$,
\begin{align}\label{eq:asympMBM2}
\frac{\EE[|X_t-X_s|^2]}{|t-s|^{2\alpha}} \sim K(t_0)\ |t-s|^{H(t)+H(s)-\alpha} + L(t_0)\ \left[\frac{H(t)-H(s)}{|t-s|^{\alpha}}\right]^2,
\end{align}
when $\rho\rightarrow 0$.
This expression allows to evaluate the exponents $\alphalocbb_X(t_0)$ (and consequently $\widetilde{\alphar}_X(t_0)$) and $\alphainfbb_X(t_0)$, in function of the respective exponents of the function $H$.

\medskip

The local behaviour of $H$ around $t_0$ is described by one of the two following situations:
\begin{itemize}
\item Either there exists $\rho>0$ such that the restriction $\left.H\right|_{B(t_0,\rho)}$ is increasing or decreasing.
In that case, $\underline{\alpha}_H(t_0)\in\R_+\cup\{+\infty\}$.

\item Or for all $\rho>0$, there exist $s,t\in B(t_0,\rho)$ such that $H(t)=H(s)$.\\
In that case, for all $\alpha>0$ and for all $\rho>0$, $\ds\inf_{s,t\in B(t_0,\rho)} \frac{|H(t)-H(s)|}{|t-s|^{\alpha}}=0$ and therefore, $\underline{\alpha}_H(t_0)=+\infty$.

\end{itemize}

Since $\widetilde{\alpha}_H(t_0) \leq \underline{\alpha}_H(t_0)$ for all $t_0\in\R_+$ as noticed in Section~\ref{sec:subexp}, we distinguish the three following cases:

\begin{enumerate}[(i)]
\item If $H(t_0) < \widetilde{\alpha}_H(t_0) \leq \underline{\alpha}_H(t_0)$ for some $t_0\in\R_+$, then for all $0<\epsilon<\widetilde{\alpha}_H(t_0) - H(t_0)$, there exists $\rho_0>0$ such that
$$ \forall t\in B(t_0,\rho_0),\quad H(t_0)-\epsilon < H(t) < H(t_0)+\epsilon,$$
and thus
\begin{equation}\label{eq:mbm-H}
\forall s,t\in B(t_0,\rho_0),\quad |t-s|^{2H(t_0)+2\epsilon} \leq |t-s|^{H(s)+H(t)} \leq |t-s|^{2H(t_0)-2\epsilon}. 
\end{equation}
Then, expression (\ref{eq:asympMBM2}) implies $H(t_0)-\epsilon\leq\widetilde{\bbalpha}_X(t_0)$ and $\underline{\bbalpha}_X(t_0) \leq H(t_0)+\epsilon$, by definition of the exponents. Letting $\epsilon$ tend to $0$, and using $\widetilde{\bbalpha}_X(t_0) \leq \underline{\bbalpha}_X(t_0)$, we get $\widetilde{\bbalpha}_X(t_0) = \underline{\bbalpha}_X(t_0) = H(t_0)$.

\noindent Then, Theorem \ref{thmain} (with $N>d\ \underline{\bbalpha}_X(t_0)$)  implies: 
\begin{align*}
\lim_{\rho\rightarrow 0}\dimH(\gr_X(B(t_0,\rho))) = 2-H(t_0) \qquad\textrm{a.s.}
\end{align*}

\item If $\widetilde{\alpha}_H(t_0) < H(t_0) \leq \underline{\alpha}_H(t_0)$ for some $t_0\in\R_+$, then as previously, we consider any $0<\epsilon<H(t_0)-\widetilde{\alpha}_H(t_0)$ and we show that expression (\ref{eq:asympMBM2}) and inequalities (\ref{eq:mbm-H}) imply $\widetilde{\bbalpha}_X(t_0) = \widetilde{\alpha}_H(t_0)$ and $\underline{\bbalpha}_X(t_0) = H(t_0)$.
Theorem \ref{thmain} (with $N>d\ \underline{\bbalpha}_X(t_0)$)  implies: 
\begin{align*}
2-H(t_0) \leq \lim_{\rho\rightarrow 0}\dimH(\gr_X(B(t_0,\rho))) \leq 2-\widetilde{\alpha}_H(t_0) \qquad\textrm{a.s.}
\end{align*}

\item If $\widetilde{\alpha}_H(t_0) \leq \underline{\alpha}_H(t_0) < H(t_0)$ for some $t_0\in\R_+$, then as previously, we consider any $0<\epsilon<H(t_0)-\underline{\alpha}_H(t_0)$ and we show that expression (\ref{eq:asympMBM2}) and inequalities (\ref{eq:mbm-H}) imply $\widetilde{\bbalpha}_X(t_0) = \widetilde{\alpha}_H(t_0)$ and $\underline{\bbalpha}_X(t_0) = \underline{\alpha}_H(t_0)$.
Theorem \ref{thmain} (with $N>d\ \underline{\bbalpha}_X(t_0)$)  implies: 
\begin{align*}
2-\underline{\alpha}_H(t_0) \leq \lim_{\rho\rightarrow 0}\dimH(\gr_X(B(t_0,\rho))) \leq 2-\widetilde{\alpha}_H(t_0) \qquad\textrm{a.s.}
\end{align*}

\end{enumerate}

Since $H$ is $\beta$-H\"older-continuous with $\beta>0$, Theorem \ref{thmainunif} can be applied with $\mathcal{A}=\R_+$. 
In the three previous case, we observe that $\underline{\bbalpha}_X(u) < 1$  for all $u\in\R_+$. Consequently, $N>d\ \liminf_{u\rightarrow t_0}\underline{\bbalpha}_X(u)$ and, with probability one,
\begin{align*}
\forall t_0\in\R_+,\quad
\lim_{\rho\rightarrow 0}\dimH(\rg_X(B(t_0,\rho))) = 1.
\end{align*}

When the $(H_{\beta})$-assumption holds, $\widetilde{\bbalpha}_X(t_0) = \widetilde{\alpha}_H(t_0)=\underline{\bbalpha}_X(t_0)$ for all $t_0\in\R_+$, and by continuity of $H$, 
$$ \liminf_{u\rightarrow t_0}\widetilde{\bbalpha}_X(u) =
\liminf_{u\rightarrow t_0}\underline{\bbalpha}_X(u) = H(t_0). $$
Then, Theorem \ref{thmainunif} implies:
With probability one,  
\begin{align*}
\forall t_0\in\R_+,\quad
\lim_{\rho\rightarrow 0}\dimH(\gr_X(B(t_0,\rho))) = 2-H(t_0).
\end{align*}

\end{proof}

According to Proposition \ref{prop:mbm}, the general theorems of Section \ref{sec:main} fail to derive sharp values for the Hausdorff dimensions of the sample paths of the multifractional Brownian motion when the $(H_{\beta})$-assumption for the function $H$ is not satisfied.
This is due to the fact that the irregularity of $H$ is not completely controlled by the exponents $\widetilde{\alpha}_H(t_0)$ and $\underline{\alpha}_H(t_0)$. A deeper analysis of the function $H$ is required in order to determine the exact Hausdorff dimensions of the mBm.

\medskip

\subsection{Generalized Weierstrass function}\label{sec:GW}

The local regularity of the Weierstrass function $W_H$, defined by
\begin{equation*}
t\mapsto W_H(t)=\sum_{j=1}^{\infty}\lambda^{-j H}\ \sin\lambda^{j}t,
\end{equation*}
where $\lambda \geq 2$ and $H\in (0,1)$, has been deeply studied in the literature (e.g. see \cite{falconer}). 
When $\lambda$ is large enough, the box-counting dimension of the graph of $W_H$ is known to be $2-H$.
Nevertheless the exact value of the Hausdorff dimension remains unknown at this stage.

Different stochastic versions of the Weierstrass function have been considered in \cite{AyLeVe, falconer, 2ml, hunt, liningPhD} and their geometric properties have been investigated.
In this section, we consider the {\em generalized Weierstrass function (GW)}, defined as the Gaussian process $X=\left\{X_t;\;t\in\mathbf{R}_{+}\right\}$,
\begin{equation}\label{def:Weierstrass}
\forall t\in\R_+,\quad
X_t=\sum_{j=1}^{\infty}Z_j\ \lambda^{-j H(t)}\ \sin(\lambda^{j}t + \theta_j)
\end{equation}
where 
\begin{itemize}
\item $\lambda \geq 2$,
\item $t\mapsto H(t)$ takes values in $(0,1)$,
\item $\left(Z_j\right)_{j\geq 1}$ is a sequence of $\mathcal{N}(0,1)$ i.i.d. random variables, 
\item and %either $\theta_j=0$ for all $j\geq 1$, or 
$\left(\theta_j\right)_{j\geq 1}$ is a sequence of uniformly distributed on $[0,2\pi)$ random variables independent of $\left(Z_j\right)_{j\geq 1}$. 
\end{itemize}

In the specific case of $\theta_j=0$ for all $j\geq 1$, Theorem 4.9 of \cite{2ml} determines the local regularity of the sample path of the GW through its $2$-microlocal frontier, when the function $H$ is $\beta$-H\"older continuous with $\beta>0$ and when the $(H_{\beta})$-assumption holds, i.e. $H(t)<\beta$ for all $t\in\R_+$. 
In particular, the deterministic local H\"older exponent is proved to be $\widetilde{\bbalpha}_X(t_0) = H(t_0)$ for all $t_0\in\R_+$ and the local H\"older exponent satisfies, with probability one,
\begin{align*}
\forall t_0\in\R_+,\quad
\widetilde{\alphar}_X(t_0) = H(t_0).
\end{align*}

Moreover, when $H$ is constant and $\theta_j=0$ for all $j\geq 1$, the Hausdorff dimension of the graph of the sample path of the GW is proved to be equal to $2-H$, as a particular case of Theorem 5.3.1 of \cite{liningPhD}. In the sequel, we use Theorem \ref{thmainunif} to extend this result when $H$ is no longer constant and the $\theta_j$'s are not equal to $0$. 

\medskip

The two following lemmas are the key results to determine the deterministic local H\"older exponent and sub-exponent of the GW, in the general case. Their proofs of are sketched in \cite{falconer} when $\left(\theta_j\right)_{j\geq 1}$ are independent and uniformly distributed on $[0,2\pi)$; for sake of completeness, we detail them in this section without requiring the independence of the $\theta_j$'s, %and consider also the case $\theta_j=0$ for all $j\geq 1$.
before considering the case of a non-constant function $H$.

\medskip

\begin{lemma}\label{lem:GW-inc-var}
Let $\{X_t;\;t\in\R_+\}$ be the stochastic Weierstrass function defined by (\ref{def:Weierstrass}). 
Then, the incremental variance between $u,v\in\R_+$ is given by
\begin{align}\label{eq:GW-inc-var}
\EE[|X_u-X_v|^2] = 2\sum_{j\geq 1} \lambda^{-2j H(u)} 
\sin^2\left(\lambda^j \frac{u-v}{2}\right) 
+ \sum_{j\geq 1} \left( \lambda^{-j H(v)} - \lambda^{-j H(u)} \right)^2.
\end{align}

\end{lemma}

\begin{proof}
For all $u,v\in\R_+$, we compute
\begin{align*}
X_u-X_v &= \sum_{j\geq 1} Z_j\ \lambda^{-j H(u)} \left[ \sin(\lambda^j u + \theta_j) - \sin(\lambda^j v + \theta_j) \right] \\
&\qquad\qquad + \sum_{j\geq 1} Z_j\ \left[ \lambda^{-j H(v)} - \lambda^{-j H(u)} \right] \sin(\lambda^j v + \theta_j) \\
&= 2 \sum_{j\geq 1} Z_j\ \lambda^{-j H(u)} \sin\left( \lambda^j \frac{u-v}{2} \right) \cos\left( \lambda^j \frac{u+v}{2} + \theta_j \right) \\
&\qquad\qquad + \sum_{j\geq 1} Z_j\ \left[ \lambda^{-j H(v)} - \lambda^{-j H(u)} \right] \sin(\lambda^j v + \theta_j).
\end{align*}

In the expression of $\EE[|X_u-X_v|^2]$, the three following terms appear:
\begin{itemize}
\item $\displaystyle\EE\left[ Z_j Z_k \cos\left( \lambda^j \frac{u+v}{2} + \theta_j \right) \cos\left( \lambda^k \frac{u+v}{2} + \theta_k \right) \right]$,
\item $\displaystyle\EE\left[ Z_j Z_k \sin\left( \lambda^j v + \theta_j \right) \sin\left( \lambda^k v + \theta_k \right) \right]$
\item and $\displaystyle\EE\left[ Z_j Z_k \cos\left( \lambda^j \frac{u+v}{2} + \theta_j \right) \sin\left( \lambda^k v + \theta_k \right) \right]$,
\end{itemize}
where $j,k \geq 1$.

The first two terms are treated in the same way. For the second one, we have
\begin{align*}
\EE\left[ Z_j Z_k \right. &\left.\sin\left( \lambda^j v + \theta_j \right) \sin\left( \lambda^k v + \theta_k \right) \right] \\
&= \EE\left( \EE\left[ Z_j Z_k \sin\left( \lambda^j v + \theta_j \right) \sin\left( \lambda^k v + \theta_k \right) \mid Z_j, Z_k\right] \right) \\
&= \EE[Z_j Z_k]\ \EE\left[ \sin\left( \lambda^j v + \theta_j \right) \sin\left( \lambda^k v + \theta_k \right) \right],
\end{align*}
using the independence of $(\theta_j, \theta_k)$ with $(Z_j, Z_k)$.
Then, since $\EE[Z_j Z_k] = \mathbbm{1}_{j=k}$ and 
\begin{align*}
\EE[\sin^2(\lambda^j v + \theta_j)] = \frac{1}{2\pi} \int_{[0,2\pi)} \sin^2(\lambda^j v + x)\ dx = \frac{1}{2},
\end{align*}
we get
\begin{align*}
\EE\left[ Z_j Z_k \sin\left( \lambda^j v + \theta_j \right) \sin\left( \lambda^k v + \theta_k \right) \right]
= \frac{1}{2}.\mathbbm{1}_{j=k}.
\end{align*}
In the same way, we prove that
\begin{align*}
\EE\left[ Z_j Z_k \cos\left( \lambda^j \frac{u+v}{2} + \theta_j \right) \cos\left( \lambda^k \frac{u+v}{2} + \theta_k \right) \right]
= \frac{1}{2}.\mathbbm{1}_{j=k}.
\end{align*}
For the third term, we compute as previously
\begin{align*}
\EE\bigg[ Z_j Z_k & \cos\left( \lambda^j \frac{u+v}{2} + \theta_j \right) \sin\left( \lambda^k v + \theta_k \right) \bigg] \\
&= \EE[Z_j Z_k] \ \EE\left[ \cos\left( \lambda^j \frac{u+v}{2} + \theta_j \right) \sin\left( \lambda^k v + \theta_k \right) \right] \\
&= \mathbbm{1}_{j=k} . \EE\left[ \cos\left( \lambda^j \frac{u+v}{2} + \theta_j \right) \sin\left( \lambda^j v + \theta_j \right) \right] \\
&= \mathbbm{1}_{j=k} . \frac{1}{2\pi}\int_{[0,2\pi)}  \cos\left( \lambda^j \frac{u+v}{2} + x \right) \sin\left( \lambda^j v + x \right) \ dx =0,
\end{align*}
by a parity argument.
The result follows.
\end{proof}

\begin{lemma}\label{lem:GW-const}
Let $\{X_t;\;t\in\R_+\}$ be the stochastic Weierstrass function defined by (\ref{def:Weierstrass}), where the function $H$ is assumed to be constant.

Then, for all compact subset $I\subset\R_+$, there exists two constants $C_1>0$ and $C_2>0$ such that for all $u,v\in I$,
\begin{align}\label{eq:GW-const}
0 < C_1 \leq \frac{\EE[|X_u-X_v|^2]}{|u-v|^{2H}} \leq C_2 < +\infty.
\end{align}

\end{lemma}

\begin{proof}
According to Lemma \ref{lem:GW-inc-var}, the incremental variance of $X$ is given by
\begin{align}\label{eq:GW-inc-var-const}
\EE[|X_u-X_v|^2] = 2\sum_{j\geq 1} \lambda^{-2j H} 
\sin^2\left(\lambda^j \frac{u-v}{2}\right).
\end{align}

%We distinguish the two cases: 
%\begin{enumerate}[(i)]
%\item $\left(\theta_j\right)_{j\geq 1}$ is a sequence of uniformly distributed on $[0,2\pi)$ random variables independent of $\left(Z_j\right)_{j\geq 1}$;
%\item $\theta_j=0$ for all $j\geq 1$.
%\end{enumerate}
%
%{\bf (i)}
Let $N$ be the integer such that $\lambda^{-(N+1)} \leq |u-v| < \lambda^{-N}$.

For all $j\leq N$, $\displaystyle\lambda^j \frac{u-v}{2} \leq \frac{1}{2}$.
Since $\displaystyle x^2 - \frac{x^4}{3} \leq \sin^2 x \leq x^2$ for all $x\in [0,1]$, expression (\ref{eq:GW-inc-var-const}) implies
\begin{align}\label{eq:GW-const-up}
\EE[|X_u-X_v|^2] &\leq 2\sum_{j=1}^N \lambda^{-2j H} \lambda^{2j} \left(\frac{u-v}{2}\right)^2 + 2\sum_{j\geq N+1} \lambda^{-2j H} \nonumber\\
&\leq 2\sum_{j=1}^N \lambda^{2j(1-H)} \left(\frac{u-v}{2}\right)^2 
+ \frac{2\ \lambda^{-2H(N+1)}}{1-\lambda^{-2H}} \nonumber\\
&\leq 2\sum_{j=1}^N \lambda^{2j(1-H)} \left(\frac{u-v}{2}\right)^2 
+ \frac{2\ |u-v|^{2H}}{1-\lambda^{-2H}}
\end{align}

and
\begin{align}\label{eq:GW-const-low}
\EE[|X_u-X_v|^2] &\geq 2\sum_{j=1}^N \lambda^{-2j H} \lambda^{2j} \left(\frac{u-v}{2}\right)^2 - \frac{2}{3} \sum_{j=1}^N \lambda^{-2j H} \lambda^{4j} \left(\frac{u-v}{2}\right)^4 \nonumber\\
%&\geq 2\sum_{j=1}^N \lambda^{2j(1-H)} \left(\frac{u-v}{2}\right)^2 
%- \frac{1}{24} \sum_{j=1}^N \lambda^{-2j H} \lambda^{4j} \lambda^{-4N} \\
&\geq 2\sum_{j=1}^N \lambda^{2j(1-H)} \left(\frac{u-v}{2}\right)^2 
- \frac{1}{24} \lambda^{-4N} \sum_{j=1}^N \lambda^{j(4-2H)} \nonumber\\
&\geq 2\sum_{j=1}^N \lambda^{2j(1-H)} \left(\frac{u-v}{2}\right)^2 
- \frac{1}{24} \lambda^{-4N} \lambda^{4-2H} \frac{\lambda^{(4-2H)N} - 1}{\lambda^{4-2H} - 1} \nonumber\\
&\geq 2\sum_{j=1}^N \lambda^{2j(1-H)} \left(\frac{u-v}{2}\right)^2 
- \frac{1}{24}  \frac{\lambda^{4-2H}}{\lambda^{4-2H} - 1} \lambda^{-2HN} \nonumber\\
&\geq 2\sum_{j=1}^N \lambda^{2j(1-H)} \left(\frac{u-v}{2}\right)^2 
- \frac{1}{24}  \frac{\lambda^{4}}{\lambda^{4-2H} - 1} |u-v|^{2H}.
\end{align}
Now, it remains to compare the term $\displaystyle\sum_{j=1}^N \lambda^{2j(1-H)} \left(\frac{u-v}{2}\right)^2$ with $|u-v|^{2H}$.

By definition of the integer $N$, we have
\begin{align}\label{eq:GW-const-ineq-1er}
\frac{\lambda^{-2(N+1)}}{4} \sum_{j=1}^N \lambda^{2j(1-H)} 
\leq \sum_{j=1}^N \lambda^{2j(1-H)} \left(\frac{u-v}{2}\right)^2
\leq \frac{\lambda^{-2N}}{4} \sum_{j=1}^N \lambda^{2j(1-H)}.
\end{align}
But
\begin{align*}
\frac{\lambda^{-2N}}{4} \sum_{j=1}^N \lambda^{2j(1-H)}
&= \frac{\lambda^{-2N}}{4} \lambda^{2(1-H)} \frac{\lambda^{2N(1-H)}-1}{\lambda^{2(1-H)}-1} \\
&= \frac{\lambda^{2(1-H)}}{4(\lambda^{2(1-H)}-1)}  \left(\lambda^{-2NH}-\lambda^{-2N}\right).
\end{align*}

Using the definition of $N$, we get
\begin{align*}
|u-v|^{2H} - \lambda^2\ |u-v|^2
\leq \lambda^{-2NH}-\lambda^{-2N} \leq
\lambda^{2H}\ |u-v|^{2H} - |u-v|^2.
\end{align*}
Then there exists two constants $c_1>0$ and $c_2>0$ such that for all $u,v\in I$,
\begin{align*}
c_1\ |u-v|^{2H}
\leq \frac{\lambda^{-2N}}{4} \sum_{j=1}^N \lambda^{2j(1-H)} \leq
c_2\ |u-v|^{2H}.
\end{align*}
Then, the result follows from \eqref{eq:GW-const-up}, \eqref{eq:GW-const-low} and\eqref{eq:GW-const-ineq-1er}.
%{\bf (ii)}
\end{proof}

\medskip

When the function $H:\R_+\rightarrow (0,1)$ is $\beta$-H\"older continuous (and no longer constant), the double inequality (\ref{eq:GW-const}) can be improved by the following result. 

\begin{proposition}\label{prop:GW}
Let $X=\{X_t;\;t\in\R_+\}$ be a generalized Weierstrass function defined by (\ref{def:Weierstrass}), where the function $H$ is assumed to be $\beta$-H\"older-continuous with $\beta>0$.

Then, for any $t_0\in\R_+$, 
%the incremental variance of $X$ admits the asymptotic behaviour
%\begin{align*}
%\EE[|X_u-X_v|^2] \asymp |u-v|^{2H(t_0)} + [H(u) - H(v)]^2.
%\end{align*}
%More precisely, 
for all $\epsilon > 0$, there exist $\rho_0>0$ and positive constants $c_1, c_2, c_3, c_4$ such that for all $u,v\in B(t_0,\rho_0)$,
\begin{align}
& c_1\ |u-v|^{2 H(t_0)+\epsilon} + c_3\ [H(u) - H(v)]^2 \leq \EE[|X_u-X_v|^2] \label{eq:GW-min}\\
&\textrm{and} \qquad
\EE[|X_u-X_v|^2] \leq c_2\ |u-v|^{2 H(t_0)-\epsilon} + c_4\ [H(u) - H(v)]^2. \label{eq:GW-max}
\end{align}

\end{proposition}

\begin{proof}
%Following the expression (\ref{eq:GW-inc-var}) for $\EE[|X_u-X_v|^2]$, we compare the first term with $\displaystyle 2\sum_{j\geq 1} \lambda^{-2j H(t_0)}\ \sin^2\left(\lambda^j \frac{u-v}{2}\right)$, when $u$ and $v$ are close to $t_0\in\R_+$.
%
%Let 
%\begin{align*}
%\Delta &= 2\sum_{j\geq 1} \lambda^{-2j H(t_0)}\ \sin^2\left(\lambda^j \frac{u-v}{2}\right)
%- 2\sum_{j\geq 1} \lambda^{-2j H(u)}\ \sin^2\left(\lambda^j \frac{u-v}{2}\right) \\
%&= 2\sum_{j\geq 1} \left[ \lambda^{-2j H(t_0)} - \lambda^{-2j H(u)} \right]\ \sin^2\left(\lambda^j \frac{u-v}{2}\right).
%\end{align*}
%From the approximation $e^x = 1+x+O(x^2)$, we get
%\begin{align*}
%\lambda^{-2j H(t_0)} - \lambda^{-2j H(u)} &= \lambda^{-2j H(t_0)} 
%\left[ 1 - \lambda^{-2j (H(u)-H(t_0))} \right] \\
%&= \lambda^{-2j H(t_0)} \left[ 2j (H(u)-H(t_0)) + j^2.O((H(u)-H(t_0))^2) \right].
%\end{align*}
%
%Then,
%\begin{align*}
%|\Delta| \leq 2 (H(u)-H(t_0)) \sum_{j\geq 1} j \lambda^{-2j H(t_0)}
%+ O((H(u)-H(t_0))^2) \sum_{j\geq 1} j^2 \lambda^{-2j H(t_0)}
%\end{align*}
%
%
%Other proof for the first term

Since the function $H:\R_+\rightarrow (0,1)$ is continuous, for all $t_0\in\R_+$ and all $\epsilon>0$, there exists $\rho_0>0$ such that 
\begin{align*}
\forall u,v\in B(t_0,\rho_0),\quad
H(u), H(v) \in (H(t_0)-\epsilon; H(t_0)+\epsilon).
\end{align*}
Then, the first term of the expression (\ref{eq:GW-inc-var}) for $\EE[|X_u-X_v|^2]$ satisfies
\begin{align*}
2\sum_{j\geq 1} \lambda^{-2j H(u)}\ \sin^2\left(\lambda^j \frac{u-v}{2}\right) \leq
2\sum_{j\geq 1} \lambda^{-2j (H(t_0)-\epsilon)}\ \sin^2\left(\lambda^j \frac{u-v}{2}\right)
\end{align*}
and 
\begin{align*}
2\sum_{j\geq 1} \lambda^{-2j H(u)}\ \sin^2\left(\lambda^j \frac{u-v}{2}\right) \geq
2\sum_{j\geq 1} \lambda^{-2j (H(t_0)+\epsilon)}\ \sin^2\left(\lambda^j \frac{u-v}{2}\right).
\end{align*}
Then, according to Lemma \ref{lem:GW-const}, there exist two constants $c_1>0$ and $c_2>0$ such that for all $u,v\in B(t_0,\rho_0)$,
\begin{align}\label{eq:GW-first}
c_1\ |u-v|^{2(H(t_0) + \epsilon)}
\leq 2\sum_{j\geq 1} \lambda^{-2j H(u)}\ \sin^2\left(\lambda^j \frac{u-v}{2}\right) \leq
c_2\ |u-v|^{2(H(t_0) - \epsilon)}.
\end{align}

For the second term of the expression (\ref{eq:GW-inc-var}) for $\EE[|X_u-X_v|^2]$, we consider the function $\psi_{\lambda,j}:x\mapsto \lambda^{-jx}=e^{-jx \ln\lambda}$ of derivative $\psi'_{\lambda,j}(x) = -j \ln\lambda \ \lambda^{-jx}$.

From the finite increment theorem, for all $u, v\in B(t_0,\rho_0)$, there exists $h_{uv}$ between $H(u)$ and $H(v)$ (i.e. in either $(H(u),H(v))$ or $(H(v),H(u))$) such that
\begin{align*}
|\lambda^{-j H(u)} - \lambda^{-j H(v)}| 
= |H(u)-H(v)|\ j \ln\lambda \ \lambda^{-j h_{uv}}.
\end{align*}
Using the fact that $H(u)$ and $H(v)$ belong to the interval $(H(t_0)-\epsilon,H(t_0)+\epsilon)$ implies $H(t_0)-\epsilon<h_{uv}<H(t_0)+\epsilon$, we get
\begin{align*}
&|H(u)-H(v)|\ j \ln\lambda \ \lambda^{-j (H(t_0)+\epsilon)} 
\leq |\lambda^{-j H(u)} - \lambda^{-j H(v)}| \\
\textrm{and}\quad &|\lambda^{-j H(u)} - \lambda^{-j H(v)}|
\leq |H(u)-H(v)|\ j \ln\lambda \ \lambda^{-j (H(t_0)-\epsilon)}.
\end{align*}
Since $\sum_{j\geq 1}j\lambda^{-j(H(t_0)-\epsilon)} < +\infty$ and $\sum_{j\geq 1}j\lambda^{-j(H(t_0)+\epsilon)} < +\infty$, the second term of (\ref{eq:GW-inc-var}) is bounded by
\begin{align}\label{eq:GW-2nd}
c_3\ [H(u)-H(v)]^2 
\leq \sum_{j\geq 1}\left[\lambda^{-j H(u)} - \lambda^{-j H(v)}\right]^2 \leq
c_4\ [H(u)-H(v)]^2.
\end{align}
The result follows from (\ref{eq:GW-inc-var}), (\ref{eq:GW-first}) and (\ref{eq:GW-2nd}).
\end{proof}

The following result shows that Theorem \ref{thmainunif} allows to derive the Hausdorff dimensions of the graph of the generalized Weierstrass function.

\begin{corollary}\label{prop:Weierstrass}
Let $X=\{X_t;\;t\in\R_+\}$ be a generalized Weierstrass function defined by (\ref{def:Weierstrass}), where the function $H$ is assumed to be $\beta$-H\"older-continuous with $\beta>0$ and satisfies the $(H_{\beta})$-assumption.

Then, the local H\"older exponents and sub-exponents of $X$ are given by
\begin{align*}
\forall t_0\in\R_+,\quad
\widetilde{\bbalpha}_X(t_0)=\underline{\bbalpha}_X(t_0)=H(t_0).
\end{align*}

Consequently, the Hausdorff dimensions of the graph and the range of the sample path of $X$ satisfy: With probability one,  
\begin{align*}
\forall t_0\in\R_+,\quad
\lim_{\rho\rightarrow 0}\dimH(\gr_X(B(t_0,\rho))) &= 2-H(t_0), \\
\lim_{\rho\rightarrow 0}\dimH(\rg_X(B(t_0,\rho))) &= 1.
\end{align*}

\end{corollary}

\begin{proof}
According to the $(H_{\beta})$-assumption, $H(t_0) < \beta$ for all $t_0\in\R_+$.

Let us fix $t_0\in\R_+$ and consider any $0<\epsilon< 2(\beta-H(t_0))$.
From Proposition \ref{prop:GW} and the fact that $H$ is $\beta$-H\"older continuous with $2H(t_0)-\epsilon<2H(t_0)+\epsilon<2\beta$, there exist $\rho_0>0$ and two constants $C_1>0$ and $C_2>0$ such that for all $u,v \in B(t_0,\rho_0)$,
\begin{align*}
C_1\ |u-v|^{2H(t_0) + \epsilon} \leq \EE[|X_u-X_v|^2] \leq
C_2\ |u-v|^{2H(t_0) - \epsilon}.
\end{align*}

From the definitions of the deterministic local H\"older exponent and sub-exponent 
$\widetilde{\bbalpha}_X(t_0)$ and $\underline{\bbalpha}_X(t_0)$, we get
\begin{align*}
\forall 0<\epsilon< 2(\beta-H(t_0)), \quad &\widetilde{\bbalpha}_X(t_0) \geq H(t_0) - \epsilon/2, \\
&\underline{\bbalpha}_X(t_0) \leq H(t_0) + \epsilon/2
\end{align*}
and therefore, $H(t_0)\leq\widetilde{\bbalpha}_X(t_0)\leq\underline{\bbalpha}_X(t_0)\leq H(t_0)$ leads to $\widetilde{\bbalpha}_X(t_0)=\underline{\bbalpha}_X(t_0)=H(t_0)$.

Consequently, by continuity of the function $H$, Theorem \ref{thmainunif} implies: With probability one,
\begin{align*}
\forall t_0\in\R_+,\quad
\lim_{\rho\rightarrow 0}\dimH(\gr_X(B(t_0,\rho))) &= 2-H(t_0), \\
\lim_{\rho\rightarrow 0}\dimH(\rg_X(B(t_0,\rho))) &= 1.
\end{align*}

\end{proof}

\begin{remark}
Proposition \ref{prop:Weierstrass} should be compared to Theorem 1 of \cite{hunt},
where the Hausdorff dimension of the graph of the process $\{Y_t;\;t\in\R_+\}$ defined by
$$\forall t\in\R_+,\quad Y_t = \sum_{n=1}^{+\infty} \lambda^{-nH} \sin(\lambda^n t + \theta_n),$$
%$$\forall t\in\R_+,\quad Y_t = \sum_{n=0}^{+\infty} a^n \cos[2\pi (b^n t + \theta_n)],$$
%where $0<a<1<b$, $ab\geq 1$ and 
where $\lambda\geq 2$, $H\in (0,1)$ and $(\theta_n)_{n\geq 1}$ are independent random variables uniformly distributed on $[0, 2\pi)$, is proved to be $D=2-H$. %$D= 2+\log a/\log b$.

The generalized Weierstrass function $X$ differs from the process $Y$, in the form of the random serie (the $\theta_n$'s in the definition of $Y_t$ cannot be all equal) and in the fact that the exponent $H$ is constant in the definition of $Y$, on the contrary to $X$.
\end{remark}

\end{document}